\documentclass[a4paper,11pt]{article}

\usepackage{amssymb,amsmath,amsfonts,enumerate,bbm,mcode,verbatim}
\usepackage{color}
\usepackage{amsthm}
\usepackage[OT2,T1]{fontenc}
\usepackage{geometry}
\usepackage[active]{srcltx}
\usepackage[nathan]{nathanchap}
\usepackage[all]{xy}
\usepackage{graphicx}
\usepackage{float}

\newtheorem{theorem}{Theorem}[section]
\newtheorem{proposition}[theorem]{Proposition}
\newtheorem{lemma}[theorem]{Lemma}
\newtheorem{definition}[theorem]{Definition}

\newcommand{\longhookrightarrow}{{\lhook\joinrel\relbar\joinrel\rightarrow}}
\newcommand{\tmop}[1]{\ensuremath{\operatorname{#1}}}
\newcommand{\assign}{:=}

\newcommand{\longrightarrowlim}{\mathop{\longrightarrow}\limits}
\newcommand{\ls}{L_{\sigma}}
\newcommand{\os}{\mathcal{O}}
\newcommand{\oy}{\mathcal{O}_Y}
\newcommand{\s}{\sigma}
\newcommand{\p}[1]{\mathbb{P}^{#1}}
\newcommand{\mb}[1]{\mathbb{#1}}
\newcommand{\pic}{\tmop{Pic}}
\newcommand{\lrw}{\longrightarrow}
\newcommand{\Ll}{\Lambda}
\newcommand{\la}[1]{\lambda_{#1}}

\newcommand{\ti}[1]{\tilde{#1}}

\newcommand{\hh}{H^1(G,\tmop{Pic}Y)}
\newcommand{\dra}{\dashrightarrow}

\newcommand{\fn}{\mb{F}_n}

\newcommand{\br}{\tmop{Br}}

\newtheorem{varexample}[theorem]{Example}
\newenvironment{example}{\begin{varexample}\em}{\em\end{varexample}}

\newtheorem{vardefex}[theorem]{Example and Definition}

\newtheorem{varprop}[theorem]{Properties}

\newtheorem{varnotation}[theorem]{Notation}
\newenvironment{notation}{\begin{varnotation}\em}{\em\end{varnotation}}

\newtheorem{corollary}[theorem]{Corollary}

\theoremstyle{definition}
\newtheorem{remark}[theorem]{Remark}

\title{THE CONSTRUCTION OF NUMERICALLY CALABI-YAU ORDERS ON PROJECTIVE SURFACES}
\author{HUGO BOWNE-ANDERSON}

\begin{document}
\date{ }
\maketitle

\begin{abstract}
In this paper, we construct a vast collection of maximal numerically Calabi-Yau orders utilising a noncommutative analogue of the well-known commutative cyclic covering trick. Such orders play an integral role in the Mori program for orders on projective surfaces and although we know a substantial amount  about them, there are relatively few known examples.
\end{abstract}

In the following, unless explicitly stated otherwise, all schemes are integral, normal and of finite type over an algebraically closed field $k$, ${\tmop{char}(k)=0}$. A curve (respectively surface) is a scheme of dimension 1 (respectively 2) over $k$.

\section{Introduction}\label{background}

M. Artin has famously conjectured in \cite{artconj} that noncommutative surfaces will fall into one of two birational classes: those which are finite over their centre and those birationally ruled. Here we are primarily concerned with the former, ``orders over commutative surfaces'', which are not only of interest to noncommutative algebraic geometers, but have also been used with success in commutative algebraic geometry, for example in Artin and Mumford's construction of a unirational variety which is not rational \cite{artmum}.
A great deal is known about orders. For example, both del Pezzo and numerically Calabi-Yau have been classified using their ramification data (\cite{delp}, \cite{ncy}). However, relatively few explicit examples have been realised. It is to this goal that we dedicate our project: the explicit construction of orders on surfaces.

Chan introduced the noncommutative cyclic covering trick in \cite{cyc} and utilised it in constructing a body of examples of del Pezzo orders, which are noncommutative analogues of del Pezzo surfaces. In a sense, these were the easy examples and the project contained herein of constructing numerically Calabi-Yau orders, noncommutative analogues of surfaces with Kodaira dimension $0$, is far more difficult.

Given ramification data for an order on a surface $Z$, the noncommutative cyclic covering trick, in essence, entails finding (1) a commutative cyclic cover $\pi\colon Y\to Z$  with the same ramification data and (2) a line bundle on $Y$ satisfying both a particular equation and the overlap condition. The former is usually relatively straightforward and the latter very difficult, if tractable at all.
Given ramification data for an order, we know via the Artin-Mumford sequence \cite[Section 3, Theorem 1]{artmum} whether or not such an order exists but we do not know whether it is constructible via the noncommutative cyclic cover. In fact, the vast majority of orders do not arise in this form. This, along with the fact that thay are so easy to compute, is why the noncommutative cyclic covers are so special and interesting.
For example, via the Artin-Mumford sequence, we know there exist rank $4$ order on $\p2$ ramified on a sextic $S$ but this is nonconstructive. 
Utilising the noncommutative cyclic covering trick, along with implementations of the ``surjectivity of the period map'', Nikulin theory and the Strong Torelli theorem for K3 surfaces,
in Theorem \ref{thawesome}, we construct, for any given $n\in\{3,\ldots,18\}$, $2^{n-2}-1$ distinct orders on $\p2$ ramified on a sextic. These orders are of the form $A=\oy\oplus L$, where $Y$ is a K3 double cover of $\p2$ ramified on the sextic. In Section \ref{ruled}, we achieve similar results constructing numerically Calabi-Yau orders on $\p1\times\p1$ and $\mb{F}_2$. We demonstrate that those constructed on the former are birational to a certain class of orders on $\p2$ which are also constructible using the same trick. We then construct orders on surfaces ruled over elliptic curves and rational surfaces equipped with elliptic fibrations.

In Section \ref{background2}, we recall the construction of Chan's noncommutative cyclic covering trick.
In Section \ref{p2}, we recall certain facts concerning the geometry of K3 surfaces and construct numerically Calabi-Yau orders on $\p2$, utilising the ``surjectivity of the period map'', the Strong Torelli theorem and Nikulin theory in the process.
In Sections \ref{ruled} and \ref{rat_ell} respectively, we use the noncommutative cyclic covering trick to construct numerically Calabi-Yau orders on ruled surfaces and rational surfaces equipped with an elliptic fibration.

\section{Background}\label{background2}

\begin{definition}\label{ordef}
 An {\bf order}\index{order} $A$ on a scheme $Z$ is a coherent sheaf of $\os_Z$-algebras such that 
\begin{enumerate}
\item[(i)] $A$ is torsion free and
\item[(ii)] $K(A)\assign A\otimes_{\os_Z} K(Z)$ is a central simple $K(Z)-$algebra.
\end{enumerate}

\end{definition}
\begin{remark}
 The set of orders on $Z$ contained in $K(A)$ is a partially ordered set with respect to inclusion. We call an order {\bf maximal}\index{order!maximal} if it is maximal in this poset. We deal primarily with maximal orders since they are noncommutative analogues of normal schemes.
\end{remark}

One way of studying the geometry of orders is by looking at ramification data, which is defined for a normal order over a surface $Z$. The definition of normality for an order can be found in \cite[Def. 2.3]{mmpord}, while ramification data is defined in \cite[Section 2.2]{mmpord}. Note that we see in \cite{mmpord} that any maximal order is normal.
In \cite{ncy}, Chan and Kulkarni classify numerically Calabi-Yau orders on surfaces, which are the noncommutative analogues of surfaces of Kodaira dimension zero and which we now define here.

\begin{definition}\cite{ncy}\index{order!numerically Calabi-Yau}
Let $A$ be a maximal order on a surface $Z$ with ramification curves $D_i$ and corresponding ramification indices $e_i$. Then the {\bf canonical divisor}\index{order!canonical divisor} $K_A\in\tmop{Div}Z$ is defined by
\begin{eqnarray*}
 K_A & \assign & K_Z+\sum\left(1-\frac{1}{e_i}\right)D_i
\end{eqnarray*}
We say an order $A$ is {\bf numerically Calabi-Yau} if $K_A$ is numerically trivial.
\end{definition}

\subsection{The Noncommutative Cyclic Covering Trick}

Chan's noncommutative cyclic covering trick is both a noncommutative generalisation of the cyclic covering trick \cite{bpv} and an algebro-geometric generalisation of the construction of cyclic algebras. The latter are classical objects of noncommutative algebra, which Lam describes as providing a rich source of examples of rings which exhibit different ``left'' and ``right'' behaviour (\cite{lamlec}).

In order to perform the noncommutative cyclic covering trick, we first need a noncommutative analogue of a line bundle: recall that in \cite{ncp1} Van den Bergh constructs a monoidal category of
quasi-coherent $\oy$-bimodules, where $Y$ is a scheme. The invertible objects have the form $\ls$ where $L\in \tmop{Pic}Y$ and $\sigma\in\tmop{Aut}Y$.\index{invertible bimodule} One may think of this intuitively as the $\oy$-module $L$ where
the right module structure is skewed through by $\sigma$ so that $_{\oy}L\simeq L$ and $L_{\oy}\simeq\sigma^*L$ (this ``intuitive'' description is taken pretty much verbatim from \cite{del_mod}). For a rigorous definition, see \cite{cyc} or go straight to the source \cite{av_twist}.  We can tensor two invertible bimodules according to the following formula which we may take as definition:
\begin{eqnarray*}
\ls\otimes M_\tau & \simeq & (L\otimes\sigma^*M)_{\tau\sigma}
\end{eqnarray*}
In the following, we use invertible bimodules to construct orders on surfaces.
We now describe the cyclic covering trick:
let $Y$ be a scheme,
 $\sigma:Y\longrightarrow Y$ an automorphism of order $n$, $G=\langle\sigma|\sigma^n=1\rangle$ and $L\in\tmop{Pic}Y$. Let $C$ be an effective Cartier divisor and suppose there exists an isomorphism of invertible bimodules
\begin{eqnarray*}
\phi: L^{ n}_{\sigma} & \longrightarrowlim^{\sim} & \mathcal{O}_Y(-C)
\end{eqnarray*}
for some integer $n$. We also write $\phi$ for the composite morphism $L^n_{\sigma} \longrightarrowlim^{\sim} \mathcal{O}_Y(-C)\hookrightarrow\mathcal{O}_Y$ and consider it a relation on the tensor algebra
\begin{eqnarray*}
T(Y;\ls) & := & \bigoplus_{i\geq 0}\ls^i.
\end{eqnarray*}
There is a technical condition we shall need: we say the relation $\phi$ satisfies the {\bf overlap condition}\index{overlap condition} if the following diagram commutes:
\begin{displaymath}
 \xymatrix{
\ls \otimes_Y \ls^{n-1} \otimes_Y \ls \ar[r]^(0.6){\phi\otimes 1} \ar[d]^{1 \otimes \phi} & \oy \otimes_Y \ls \ar[d]^\wr\\
\ls \otimes_Y \oy \ar[r]^\sim                                                            & \ls
}
\end{displaymath}
We set $A(Y;\ls,\phi) \assign T(Y;\ls)/(\phi)$, that is, the quotient of $T(Y;\ls)$ by the relation $(\phi)$. Note that, by \cite[Prop. 3.2]{cyc}, if the relation ${\phi:\ls^n\longrightarrow\os_Y}$ satisfies the overlap condition then
\begin{eqnarray*}
A(Y;\ls,\phi) & = & \bigoplus_{i=0}^{n-1}\ls^i.
\end{eqnarray*}
This is the noncommutative cyclic cover and $A\assign (Y;\ls,\phi)$\index{$A=A(Y;\ls,\phi)$} is known as a cyclic algebra. If the relation $\phi$ is clear, we write $A=A(Y;\ls)$.

Chan's first example \cite[Example 3.3]{cyc} tells us what these noncommutative cyclic covers look like generically:

\begin{eqnarray*}
K(A) & \simeq & \frac{K(Y)[z;\sigma]}{(z^n-\alpha)}, \tmop{where} \alpha \in K(Z).\\
\end{eqnarray*}
It follows from this that $Y$ is a ``maximal commutative quotient scheme'' of $A$. One of the reasons these cyclic covers are so interesting is that one can determine their geometric properties such as the ramification with relative ease.
The following results inform us that these cyclic covers are normal orders and describe the ramification in the case where $L_\s^n\simeq\oy$ and $Y$ is a surface. For the remainder of the chapter $\pi\colon Y\to Z$ will be a cyclic cover.
\begin{theorem}\cite[Theorem 3.6]{cyc}\label{normet}
Let $Y,Z$ be quasi-projective surfaces, $\pi\colon Y\to Z$ an $n\colon 1$ cover. Let $A=A(Y;\ls,\phi)$ be a cyclic algebra arising from a relation of the form $\phi\colon\ls^n\simeq\oy$. Assuming $\phi$ satisfies the overlap condition, then $A$ is a normal order and for $C\in Z^1$, the ramification index of $A$ at $C$ is precisely the ramification index of $\pi$ above $C$.
\end{theorem}
\begin{theorem}\label{ram}\cite[Thm 3.6 and Prop. 4.5]{cyc}
 Suppose that $Y,Z$ are smooth quasi-projective surfaces and that the $n\colon 1$ quotient
map $\pi\colon Y\to Z$ is totally ramified at $D\subset Y$. Consider the cyclic algebra $A=A(Y;L_\s)$ arising
from a relation of the form $L_\s^n\simeq\oy$. Then the ramification of $A$ along $\pi(D)$ is the cyclic cover of $D$ defined by the n-torsion line bundle $L_{|D}$.
\end{theorem}
\begin{remark}
To see that the line bundle $L_{|D}\in\tmop{Pic}D$ is $n-$torsion, notice that since $\pi$ is totally ramified at $D$, $D$ is fixed by $\s$ and hence $\s^*(L_{|D})\simeq L_{|D}$. Moreover, since $L_\s^n\simeq\oy$, $(L_{|D})_\s^n\simeq\os_{D}$, which in turn means that $L_{|D}\otimes\s^*(L_{|D})\otimes\ldots\otimes(\s^*)^{n-1}(L_{|D})\simeq\os_{D}$. We conclude that $(L_{|D})^n\simeq\os_{D}$.
\end{remark}
Chan only deals with cyclic covers $\pi\colon Y\to Z$ which are totally ramified and so Theorem \ref{ram} suffices for his purposes. We shall require an analogous result concerning ramification when $\pi$ is not totally ramified. 

\begin{lemma}\label{untot}
 Assume $Y,Z$ to be smooth surfaces and that $\pi\colon Y\to Z$ is ramified at $D\subset Y$, where $D\assign\pi^{-1}(D')$, the reduced inverse image of $D'\subset Z$. We also assume that $D=\coprod_{i=1}^dD'_i$ such that the ramification index at each $D'_i$ is $m$, where $m=\frac{n}{d}$. Consider the cyclic algebra $A=A(Y;\ls)$ arising from a relation $L_\s^n\simeq\oy$. Then the ramification of $A$ along $D'$ is the cyclic cover of $D'$ defined by the m-torsion line bundle $(L\otimes\s^*L\otimes\ldots\otimes\s^{*d}L)_{|D'_i}$, for any $i\in\{1,\ldots,d\}$.

\begin{proof}

The proof is the same method as in that of Theorem \ref{ram} and can be found in \cite[Lemma 2.10]{bowne}.
\end{proof}

\end{lemma}

The following lemma will allow us to verify when the cyclic covering trick produces maximal orders.
\begin{lemma}\label{max}
A cyclic algebra $A$ on $Z$ is maximal if for all irreducible components $D_i$ of the ramification locus $D$, the cyclic cover corresponding to $L_{|D_i}$ is irreducible.
\begin{proof}
See \cite[Lemma 2.11]{bowne}.
\end{proof}
\end{lemma}

To construct these cyclic algebras, we are interested in finding invertible bimodules $L_\s$ such that $\ls^n\simeq\oy(-C)$. In this paper, we are interested in relations of the form $L_\s^n\longrightarrowlim^\sim\oy$, to which end we define $Rel_{io}$ to be the set of all (isomorphism classes of) relations $\phi:L_\s^n\longrightarrowlim^\sim\oy$ which are isomorphisms and satisfy the overlap condition. The tensor product endows $Rel_{io}$ with an abelian group structure as follows: given two relations $\phi\colon\ls^n\to\oy$, $\psi\colon M_{\sigma}^n\to\oy$, we define their product to be the relation
\begin{eqnarray*}
\phi\otimes\psi\colon (L\otimes_{\oy}M)_\s^n\longrightarrowlim^\sim \ls^n\otimes M_{\sigma}^n\longrightarrowlim^{\phi\otimes\psi}\oy\otimes\oy=\oy
\end{eqnarray*}
There exists a subgroup $E$ of $Rel_{io}$ which will soon play an important role. Its definition is technical and we need not state it here. It suffices to say that elements of $Rel_{io}$\index{$Rel_{io}$} in the same coset of $E$ result in Morita equivalent cyclic algebras, that is, cyclic algebras with equivalent module categories. We refer the reader to the discussion following Corollary 3.4 of \cite{cyc}. 
\begin{remark}
 From the discussion preceding \cite[Lemma 3.5]{cyc}, we see that relations of the form ${\ls^n\simeq\oy}$ can be classified using cohomology. We consider $\pic Y$ as a $G-$set and recall that since $G$ is cyclic, the group cohomology of any $G-$set $M$ can be computed as the cohomology of the periodic sequence
\begin{eqnarray*}
\ldots\longrightarrowlim^N M \longrightarrowlim^D M\longrightarrowlim^N M \longrightarrowlim^D\ldots
\end{eqnarray*}
where $N=(1+\s+\ldots\s^{n-1})$ and $D=(1-\s).$
Thus 1-cocycles of the $G$-set $\tmop{Pic}Y$ are precisely invertible bimodules $\ls$ such that ${\ls^n\simeq\oy}$. Moreover, from \cite[Lemma 3.5]{cyc}, we have a group homomorphism $f \colon Rel_{io}/E \to H^1(G,\tmop{Pic}Y)$ which sends a relations $\phi \colon \ls^n\simeq \oy$ to the class $\lambda \in \hh$ of the $1$-cocycle $L$.
\end{remark}

The following proposition allows us to verify easily in certain cases that relations satisfy the overlap condition.
\begin{proposition}\cite[Prop. 4.1]{cyc}\label{ol}
Suppose that $Y$ is smooth and quasi-projective, $\pi\colon Y \to Z$ is $n\colon 1$ and $\os(Y)^*=k^*$ (this last condition holds, for example, if $Y$ is a projective variety), and the lowest common multiple of the ramification indices of $\pi\colon Y\to Z$ is $n$. Then all relations constructed from elements of $H^1(G,\tmop{Pic}Y)$ satisfy the overlap condition.
\end{proposition}
The following remark will ensure that many of the orders constructed herein are nontrivial in $\tmop{Br}(K(Z))$.
\begin{remark}\cite[Cor 4.4]{cyc}\label{inbr}
Assuming $Y$ to be smooth and projective, if $\pi\colon Y\to Z$ is totally ramified at an irreducible divisor $D\subset Y$, we have an embedding
\begin{eqnarray*}
 \Psi\colon\hh & \hookrightarrow & \tmop{Br}(K(Y)/K(Z))
\end{eqnarray*}
given by the following: let $L\in\tmop{Pic}Y$ represent a $1$-cocycle, $\phi$ the corresponding relation, which is unique up to scalar multiplication. Then $\Psi(L)\assign K(A(Y;\ls,\phi))\in\tmop{Br}(K(Y)/K(Z))$.
\end{remark}

\section{Numerically Calabi-Yau Orders on $\p2$}\label{p2}

In this section, unless explicitly stated otherwise, all schemes are defined over $\mathbb{C}$ and all sheaf cohomology is complex analytic.
In the following, we construct examples of numerically Calabi-Yau orders on rational surfaces, beginning with $\p2$. 

The first interesting numerically Calabi-Yau orders on $\p2$ are those ramified on a smooth sextic with ramification index 2. In the following, we construct {\bf quaternion orders} (orders of rank $4$)\index{order!quaternion} on $\p2$ with the desired ramification.

\subsection{Constructing orders ramified on a sextic}\label{sextic}

We wish to construct quaternion orders on $\p2$ ramified on a sextic $C$. In \cite[Example 9.2]{cyc}, Chan does this using a K3 double cover of $\p2$ ramified on $C$ and we wish to take the same approach. However, our plan of attack is to reverse engineer the $K3$ surface. We first construct a $K3$ surface $Y$ with a desired Picard lattice and then construct an automorphism $\tau$ of $\tmop{Pic}Y$ which we show is induced by an automorphism $\s$ of $Y.$ We shall make our choices such that $H^1(G,\tmop{Pic}Y)$ is non-trivial, where $G=\langle\s|\s^2=1\rangle,$ and this will allow us to construct orders on $Y/G$, which by construction will be the projective plane. In contrast to Chan's approach, we shall be able to explicitly compute $H^1(G,\tmop{Pic}Y)$ by giving $1-$cocycles which generate the group. To this end, we necessarily digress on the $K3$ lattice $\Ll$ and related phenomena.

\subsubsection{K3 surfaces}
The following results concerning K3 surfaces can be found in \cite{bpv}, unless stated otherwise. Recall that a {\bf K3 surface}\index{surface, K3} $Y$ is defined to be a smooth, projective surface with trivial canonical bundle such that $h^1(Y,\oy)=0$\index{K3 surface}. 
Triviality of the canonical bundle implies there exists a nowhere vanishing holomorphic $2-$form $\omega_Y$, unique up to multiplication by a scalar and known as the {\bf period}\index{K3 surface! period of}\index{$\omega_Y$} of $Y$.
We are interested in the K3 lattice $\Ll$, which we now define.

\begin{definition}\index{lattice! K3}
For any K3 surface $Y$, $H^2(Y,\mb{Z})\simeq\mb{Z}^{22}$ and, equipped with the cup product, is a lattice isomorphic to
\begin{eqnarray*}
\Lambda & := & \mb{E}\perp\mb{E}\perp\mb{H}\perp\mb{H}\perp\mb{H},
\end{eqnarray*}

where $\mb{E}\simeq\mb{Z}^8$ with bilinear form given by the matrix
$$
\left( \begin{array}{cccccccc}
-2 &  &  & 1 &\\
 & -2 & 1\\
 & 1 & -2 & 1\\
1 &  &  1 & -2 & 1\\
 &  &  & 1 & -2 & 1\\
 &  &  &  &  1 & -2 & 1\\
 &  &  &  &  &  1 & -2 & 1\\
 &  &  &  &  &  &  1 & -2\\
\end{array} \right)
$$

and $\mb{H}\simeq\mb{Z}^2$ with bilinear form given by

$$
\left( \begin{array}{cc}
0 & 1\\
1 & 0
\end{array} \right).
$$
We call $\Ll$ the {\bf K3 lattice}, $\mb{H}$ the {\bf hyperbolic plane}\index{lattice!hyperbolic plane}. In the following, we let $\{\lambda_1,\ldots,\lambda_8\}$ and $\{\lambda'_1,\ldots,\lambda'_8\}$ generate the first and second copies of $\mb{E}\subset\Ll$ respectively. We also let $\{\mu_1,\mu_2\}$ be generators for the first copy of $\mb{H},$ $\{\mu'_1,\mu'_2\}$ generators of the second and $\{\mu''_1,\mu''_2\}$ generators of the third.
\end{definition}

The fact that $K_Y$ is trivial for a K3 surface $Y$ makes studying the geometry of K3 surfaces far simpler than would otherwise be. The following results, concerning line bundles and curves on K3 surfaces, demonstrate this and will be made use of shortly.

\begin{proposition}\label{rr}
 Let $L$ be a line bundle on a K3 surface $Y$ such that $L\cdot L\geq-2$. Then either $L$ or $L^{-1}$ is an effective class, that is, either $h^0(L)>0$ or $h^0(L^{-1})>0$. Moreover, if $L\cdot L=-2$ and $h^0(L)\geq 0$, there is a unique effective divisor $D$ such that $L\simeq \oy(D)$.
\begin{proof}
This is \cite[Chap. VIII, Prop 3.6]{bpv}.
\end{proof}
\end{proposition}

\begin{remark}\label{k3curve}
 For an irreducible smooth curve $C$ on a K3 surface, the adjunction formula for curves on a surface yields $g(C)=C^2/2+1$. Thus $C^2\geq -2$ and if $C^2=-2$, $C$ is necessarily rational. In this case we call $C$ a {\bf nodal curve}\index{nodal curve} since we can blow it down, but only at the expense of creating a nodal singularity.
\end{remark}

The following result, which is a formulation of the "surjectivity of period map"(see \cite[Chap. VIII, Sec. 14]{bpv}), will allow us to construct $K3$ surfaces with our desired Picard lattices:
\begin{proposition}\label{mor}\cite[Cor. 1.9]{mor84}
Let $S\longhookrightarrow\Ll$ be a primitive sublattice with signature $(1,\rho-1),$ where $\rho=\tmop{rank}(S).$ Then there exists a K3 surface $Y$ and an isometry ${\tmop{Pic}Y\simeq S}.$ 
\end{proposition}
\begin{remark}\label{prim}
If $S$ is a direct summand of $\Ll$, then $S$ is a primitive sublattice. This observation will prove invaluable in applying the above Proposition \ref{mor} to verify the existence of certain K3 surfaces.
\end{remark}

\begin{example}\label{s_2}
Set $S=\mb{Z}^3=\langle s_1,s_2,s_3\rangle$ with bilinear form given by
$$
\left( \begin{array}{ccc}
-2 & 3 & 0\\
3 & -2 & 1\\
0 & 1  & -2
\end{array}\right).
$$
We can embed $S$ in $\Ll$ via
\begin{align*}
\gamma: S & \longhookrightarrow  \Ll\\
        s_1& \longmapsto          \la1+\mu_1\\
	s_2& \longmapsto          \la2+3\mu_2\\
	s_3& \longmapsto         \la3.
\end{align*}Since $\{\gamma(s_1),\gamma(s_2),\gamma(s_3),\la4,\ldots,\la8,\la1',\ldots,\la8',\mu_1,\mu_2,\mu_1',\mu_2',\mu_1'',\mu_2''\}$ is a basis for $\Ll$, $\gamma$ is a primitive embedding by Remark \ref{prim}. Moreover, the signature of $S$ is $(1,2)$ (Maple performed this calculation), implying by Proposition \ref{mor} that there exists a $K3$ surface $Y$ and an isometry $\tmop{Pic}Y\simeq S.$ Since $s_i^2=-2,$ Proposition \ref{rr} tells us that for each $i$, either $s_i$ is an effective class or $-s_i$ is. We assume without loss of generality that $s_1$ is effective. Then since $s_1\cdot s_2=3$ and $s_2\cdot s_3=1,$ $s_2$ and $s_3$ are necessarily effective classes. Similarly, $(s_1+s_2-s_3)^2=-2$ and one can deduce that $s_1+s_2-s_3$ is an effective class.  We shall soon see in Remark \ref{bita} that $Y$ is a double cover of $\p2$ ramified on a sextic $C$ with two tritangents.
\end{example}

\begin{notation}
 For the rest of this chapter, given a class of line bundles $s_i$ on a surface $Y$, we shall abuse notation by using $s_i$ to mean a representative line bundle of this class. For an effective class $s_i$, we shall also let $S_i$ be an effective divisor such that $s_i\simeq\oy(S_i)$.
\end{notation}

We shall need to retrieve an ample line bundle on $Y$ and the following lemma will allow us to do so.

\begin{lemma}\label{ample}
 Let $Y$ be a K3 surface and $\tmop{Pic}Y=\langle s_1,\ldots,s_n\rangle$, where the $s_i$ are effective classes. Let $s\in\tmop{Pic}Y$ be an effective class such that $h^0(s-s_i)>0$, for all $i$. Assume that $s^2>0$ and for all $i$, the following hold: $s\cdot s_i> 0$ and $s\cdot(s-s_i)> 0$. Then $s$ is ample.

\begin{proof}
By the Nakai-Moishezon criterion \cite[Chap. V, Thm 1.10]{ag}, $s$ is ample if and only if both $s^2>0$ and $s\cdot C> 0$ for all irreducible curves $C$ on $Y$. The first condition is one of our hypotheses so we need only verify the second. 
Since we have assumed that $s\cdot s_i> 0$ and $s\cdot(s-s_i)> 0$ for all $i$, we need only show that $s\cdot C>0$ for $C\nsim s_i$ and ${C\nsim(s-s_i)}$. Such a $C$ is an effective class distinct from the effective classes $s_i$ and $s-s_i$, implying $C\cdot(s-s_i)\geq 0$ and $C\cdot s_i\geq 0$ for all $i$. Moreover, since the $s_i$ generate $\tmop{Pic}Y$, there exists a $j$ such that $C\cdot s_j>0$ (otherwise $C\cdot D=0$ for all $D\in\tmop{Pic}Y$, implying $C\sim0$). Then
\begin{eqnarray*}
 C\cdot s & = & C\cdot s_j+C\cdot(s-s_j)\\
                  & > & 0.
\end{eqnarray*}
By the Nakai-Moishezon criterion, $s$ is an ample class.

\end{proof}
\end{lemma}

\begin{remark}\label{nodal}
We now return to the setting of Example \ref{s_2}: for $i\in\{1,2,3\}$, each $S_i$ is unique by Proposition \ref{rr}. We now show that each $S_i$ is nodal: since $s=s_1+s_2$ satisfies the conditions of Lemma \ref{ample}, $s$ is ample. Moreover, since $s\cdot s_i=1$, for all $i$, given the fact that an ample divisor will intersect an effective class strictly positively, each $S_i$ is irreducible and thus nodal. Similarly $S_4\simeq S_1+S_2-S_3$ is a nodal class.
\end{remark}

\begin{example}\label{orsol_mat}
Example \ref{s_2} is a specific instance of the following: for $n\in\{3,\ldots,18\}$, there exists a K3 surface $Y$ with Picard lattice isomorphic to $S=\mb{Z}^{n}=\langle s_1,\ldots,s_n\rangle$, the intersection form given by the $n\times n$ submatrix $M$ which is formed by the first $n$ rows and the first $n$ columns of

{\setlength\arraycolsep{2pt}
$$
Q = \left( \begin{array}{cccccccccccccccccc}
    -2   &  3  &   0   &  1  &   1  &   1  &   1  &   1  &   1   &  1  &   1   &  1  &   1   &  1  &   1  &   1  &   1  &   1\\
     3   & -2  &   1   &  0  &   0  &   0  &   0  &   0  &   0   &  0  &   0   &  0  &   0   &  0  &   0  &   0  &   0  &   0\\
     0   &  1  &  -2   &  1  &   0  &   0  &   0  &   0  &   0   &  0  &   0   &  0  &   0   &  0  &   0  &   0  &   0  &   0\\
     1   &  0  &   1   & -2  &   1  &   0  &   0  &   0  &   0   &  0  &   0   &  0  &   0   &  0  &   0  &   0  &   0  &   0\\
     1   &  0  &   0   &  1  &  -2  &   1  &   0  &   0  &   0   &  0  &   0   &  0  &   0   &  0  &   0  &   0  &   0  &   0\\
     1   &  0  &   0   &  0  &   1  &  -2  &   1  &   0  &   0   &  0  &   0   &  0  &   0   &  0  &   0  &   0  &   0  &   0\\
     1   &  0  &   0   &  0  &   0  &   1  &  -2  &   1  &   0   &  0  &   0   &  0  &   0   &  0  &   0  &   0  &   0  &   0\\
     1   &  0  &   0   &  0  &   0  &   0  &   1  &  -2  &   0   &  0  &   0   &  0  &   0   &  0  &   0  &   0  &   0  &   0\\
     1   &  0  &   0   &  0  &   0  &   0  &   0  &   0  &  -2   &  0  &   0   &  1  &   0   &  0  &   0  &   0  &   0  &   0\\
     1   &  0  &   0   &  0  &   0  &   0  &   0  &   0  &   0   & -2  &   1   &  0  &   0   &  0  &   0  &   0   &  0   &  0\\
     1   &  0  &   0   &  0  &   0  &   0  &   0  &   0  &   0   &  1  &  -2   &  1  &   0   &  0  &   0  &   0  &   0  &   0\\
     1   &  0  &   0   &  0  &   0  &   0  &   0  &   0  &   1   &  0  &   1   & -2  &   1   &  0  &   0  &   0  &   0  &   0\\
     1   &  0  &   0   &  0  &   0  &   0  &   0  &   0  &   0   &  0  &   0   &  1  &  -2   &  1  &   0  &   0  &   0  &   0\\
     1   &  0  &   0   &  0  &   0  &   0  &   0  &   0  &   0   &  0  &   0   &  0  &   1   & -2  &   1  &   0  &   0  &   0\\
     1   &  0  &   0   &  0  &   0  &   0  &   0  &   0  &   0   &  0  &   0   &  0  &   0   &  1  &  -2  &   1  &   0  &   0\\
     1   &  0  &   0   &  0  &   0  &   0  &   0  &   0  &   0   &  0  &   0   &  0  &   0   &  0  &   1  &  -2  &   0  &   0\\
     1   &  0  &   0   &  0  &   0  &   0  &   0  &   0  &   0   &  0  &   0   &  0  &   0   &  0  &   0  &   0  &  -2  &   0\\
     1   &  0  &   0   &  0  &   0  &   0  &   0  &   0  &   0   &  0  &   0   &  0  &   0   &  0  &   0  &   0  &   0  &  -2
\end{array}\right)
$$
}
We see this below in Theorem \ref{thawesome}. 

Now given a K3 surface with such a Picard lattice, we may assume that the all the $s_i$ are effective just as in Example \ref{s_2}. We now show that $s_1+s_2$ is ample: firstly, $(s_1+s_2)^2=2>0$ and
elementary computations yield that for all $i$, $(s_1+s_2)\cdot s_i=1$, $(s_1+s_2)\cdot (s_1+s_2-s_i)=1$ and $h^0(s-s_i)>0$. 
Thus $s_1+s_2$ satisfies the conditions of Lemma \ref{ample} and is consequently ample. The same argument as in Remark \ref{nodal} implies that the $s_i$ are all effective nodal classes, as are the $s_1+s_2-s_i$.
\end{example}

Given a surface $Y$ as in Example \ref{orsol_mat}, we would like to construct an automorphism of $Y$ by giving an isometry of $H^2(Y,\mb{Z})$ and thus the question we need to answer is: given an isometry $H^2(Y,\mb{Z})\lrw H^2(Y,\mb{Z}),$ how can we tell if it is induced by an automorphism $\s:Y\lrw Y?$ The Strong Torelli theorem aids us in this.
Before we introduce the Strong Torelli theorem, we need to define an effective Hodge isometry. To this end, recall that there is a Hodge decomposition of $H^2(Y,\mb{C})$:
\begin{eqnarray*}\label{hodge}
H^2(Y,\mb{C}) & \simeq & H^{0,2}(Y)\oplus H^{1,1}(Y)\oplus H^{2,0}(Y),
\end{eqnarray*}
where $H^{p,q}(Y)\simeq H^q(Y,\Omega^p)$\index{$H^{p,q}(Y)$}. Note that $\tmop{Pic}Y=H^2(Y,\mb{Z})\cap H^{1,1}(Y)$ (this is  a classical theorem of Lefschetz; see \cite[Sec. 1.3]{k3en} for details) and that $\omega_Y\in H^{2,0}(Y)\subset T_Y$.

\begin{definition} Let $Y,Y'$ be surfaces. An isometry of lattices
\begin{eqnarray*}
H^2(Y,\mb{Z}) & \lrw & H^2(Y',\mb{Z})
\end{eqnarray*}
is called an \textbf{ effective Hodge isometry}\index{effective Hodge isometry} if its $\mb{C}-$linear extension

\begin{enumerate}
 \item sends $H^{2,0}(Y)$ to $H^{2,0}(Y')$ and
\item maps the class of some ample divisor on $Y$ to the class of an ample divisor on $Y'$.
\end{enumerate}

\end{definition}

This definition is not entirely standard but it is an equivalent formulation and can be found in \cite[Introduction]{torelli}. We use it because it suits our purposes nicely. We are now able to state the Strong Torelli theorem \cite[Ch. VIII, Thm 11.1]{bpv}.

\begin{theorem}[Strong Torelli theorem]\label{tor}\index{Strong Torelli theorem}
Let $Y$ and $Y'$ be two K3 surfaces such that there exists an effective Hodge isometry $\phi\colon H^2(Y,\mb{Z})\to H^2(Y',\mb{Z})$. Then there also exists a unique biholomorphic $\s:Y'\lrw Y$ such that $\phi=\s^*.$
\end{theorem}

We would also like to determine the quotients of K3 surfaces by certain involutions. To this end, we shall require the following definition and proposition, which form part of what is now known as Nikulin theory).

\begin{definition}\index{involution! symplectic}\index{involution! anti-symplectic}
An involution $\psi$ on a K3 surface $Y$ is called {\bf symplectic} if $\psi^*(\omega_Y)=\omega_Y$. It is called {\bf anti-symplectic} if ${{\psi^*(\omega_Y)=-\omega_Y}}$.
\end{definition}

\begin{proposition}\cite[Prop.1.11]{k3sym}\label{k3sym}
 Let $\pi:Y\to Y/G$ be the quotient of a K3 surface by an anti-symplectic involution $\s.$ If Fix$_Y(\s)\neq\varnothing,$ then Fix$_Y(\s)$ is a disjoint union of smooth curves and $Y/G$ is a smooth, projective rational surface. Furthermore, Fix$_Y(\s)=\varnothing$ if and only if $Y/G$ is an Enriques surface.
\end{proposition}

\begin{remark}\label{irred}

 From \cite[Theorem 1.12]{k3sym}, if $Y/G$ in Proposition \ref{k3sym} is rational, then either 
\begin{enumerate}
 \item Fix$(\s)=\cup_i R_i\cup D_g$ (where the $R_i$ are smooth disjoint nodal curves and $D_g$ is a smooth curve of genus $g$) or
\item Fix$(\s)=D'\cup D''$ (where $D',D''$ are linearly equivalent elliptic curves).
\end{enumerate}
In case (2), $Y/G$ is both rational and elliptically fibred over $\p1$ (this is from the proof of \cite[Theorem 1.12]{k3sym}).

\end{remark}

Finally, we would like to determine the structure of the Picard group of $Y/G$. We do so in the following lemma.

\begin{lemma}\label{inj}
 Let $\pi\colon Y\to Z$ be a double cover of a smooth, projective rational surface. Then ${\pi^*\colon\tmop{Pic}Z\to\tmop{Pic}Y}$ is an injection. Moreover, if the ramification locus of $D\subset Y$ is irreducible, $\pi^*$ surjects onto $(\tmop{Pic}Y)^G$.
\begin{proof}
This is \cite[Lemma 3.20]{bowne}.

\end{proof}
\end{lemma}

\begin{corollary}\label{intform}
 Let $\pi\colon Y\to Z$ be as above and $Y$ a K3 surface. Then there is an isometry of lattices $\tmop{Pic}Z\simeq \frac{1}{2}(\tmop{Pic}Y)^G$.
\begin{proof}
 Firstly, we know that $\tmop{Pic}Z\simeq (\tmop{Pic}Y)^G$. Then from \cite[Prop. I.8(ii)]{bo}, we see that for any $L_1,L_2\in\tmop{Pic}Z$, ${\pi^*L_1\cdot\pi^*L_2=2(L_1\cdot L_2)}$ and the result follows directly.
\end{proof}

\end{corollary}

We have now developed the theory necessary to construct our orders.
\begin{theorem}\label{thawesome}
Let $n\in\{3,\ldots,18\}$. Then 
\begin{enumerate}[i)]
 \item there exists a K3 surface $Y$ with Picard lattice isomorphic to $S\simeq\mb{Z}^{n}=\langle s_1,\ldots,s_n\rangle$, the intersection form given by the $n\times n$ submatrix $M$ 	which is formed by the first $n$ rows and the first $n$ columns of the matrix $Q$ in Example \ref{orsol_mat}.
	Further, there exists an involution $\s\colon Y\to Y$ such that the corresponding quotient morphism ${\pi\colon Y\to Z\assign Y/G}$ is the double cover of $\p2$ ramified on a smooth sextic $C$;
\item ${\hh\simeq (\mb{Z}/2\mb{Z})^{n-2}}$, generated by $L_i=s_1-s_i$, for $i\in\{3,\ldots,n\}$, and all relations satisfy the overlap condition.
%
Then, for $m_i\in\{0,1\}$ (not all zero), the corresponding $2^{n-2}-1$ orders $A(Y;(\otimes_{i=3}^nL_i^{m_i})_{\s})=\oy\oplus(\otimes_{i=3}^nL_i^{m_i})_{\s}$ are maximal orders on $\p2$ ramified on $C$, distinct in $\tmop{Br}(K(Z))$.
\end{enumerate}

\end{theorem}
\begin{proof}
\begin{enumerate}[i)]
 \item

Ignoring terms $s_j$ with $j>n$, we embed $S=\langle s_1,\ldots,s_{n}\rangle$ in $\Ll$ via
 \begin{equation*}
\begin{aligned}
\gamma\colon s_1 & \mapsto\la1+\mu_1, & s_2 & \mapsto\la2+3\mu_2, & s_3& \mapsto \la3,\\ 
s_4& \mapsto \la4, &s_5& \mapsto \la5+\mu_2, & s_6& \mapsto \la6+\mu_2,\\
s_7& \mapsto \la7+\mu_2, & s_8& \mapsto \la8+\mu_2, & s_9& \mapsto \la1'+\mu_2,\\
s_{10}& \mapsto \la2'+\mu_2, & s_{11}& \mapsto \la3'+\mu_2, &s_{12}& \mapsto \la4'+\mu_2,\\
s_{13}& \mapsto \la5'+\mu_2, & s_{14}& \mapsto \la6'+\mu_2, & s_{15}& \mapsto \la7'+\mu_2,\\
s_{16}& \mapsto \la8'+\mu_2, & s_{17}& \mapsto \mu_2+\mu_1'-\mu_2', & s_{18}& \mapsto \mu_2+\mu_1''-\mu_2''.\\
\end{aligned}
\end{equation*}
Since $\{\gamma(s_1),\ldots,\gamma(s_{18}),\mu_1,\mu_2,\mu_1',\mu_1''\}$ is a basis of $\Ll$, by Remark \ref{prim}, $\gamma$ is a primitive embedding. Also, $S$ has signature $(1,n-1)$ (this calculation was performed by Maple). By Proposition \ref{mor}, there is a K3 surface $Y$ and an isometry $\tmop{Pic}Y\simeq S$.  We define an isometry $\phi$ on $T_Y\oplus\tmop{Pic}Y$ as follows: for $t\in T_Y,\phi(t)=-t$; on $\tmop{Pic}Y, \phi(s_i)=s_1+s_2-s_i,i\in\{1,\ldots,n\}$. This $\phi$ extends to an isometry on $\Ll$ if and only if it preserves the integral lattice.
 MATLAB verifies that this is so (the MATLAB code for this can be found in \cite[Appendix B]{bowne}), implying that $\phi$ extends to an isometry of $H^2(Y,\mb{Z})$, which we also denote $\phi$. Now we show that $\phi$ is an effective Hodge isometry: firstly, $H^{2,0}(Y)\subset T_Y$ and $\phi(t)=-t,$ for all $t\in T_Y$, imply that $\phi$ preserves $H^{2,0}(Y)$. Since $s_1+s_2$ is both fixed by $\phi$ and ample (the latter fact demonstrated in Example \ref{orsol_mat}), $\phi$ is an effective Hodge isometry and there exists an involution $\s\colon Y\to Y$ such that $\phi=\s^*$ by the Strong Torelli theorem.
Since $\omega_Y\in T_Y$ and $\s(t)=-t$ for all $t\in T_Y$, $\s$ is antisymplectic. Thus $Z=Y/G$ (here $G=\langle\s|\s^2=1\rangle$) is a smooth rational surface or an Enriques surface by Proposition \ref{k3sym}. Recalling that $S_i$ is the divisor such that $s_i\simeq\oy(S_i)$, $\pi_{|S_1\cup S_2}$ is the double cover of $\p1$ by two copies of $\p1$ intersecting in three points. Thus Fix$_Y(\s)\neq\varnothing$ and by Proposition \ref{k3sym}, $Z$ is a smooth rational surface. 
By Lemma \ref{inj}, $\pi^*\colon \tmop{Pic}Z \to (\tmop{Pic}Y)^G\simeq\mb{Z}(s_1+s_2)$ is injective, implying $\tmop{Pic}Z\simeq\mb{Z}$, from which we conclude that $Z\simeq\p2$. The Hurwitz formula ${\omega_Y=\pi^*(\omega_Z)\otimes\oy(R)}$, where $R$ is the ramification divisor, along with the fact that $\omega_Y\simeq\oy$ imply that $\pi$ is ramified on a sextic $C$. By Proposition \ref{k3sym}, $C$ is smooth and thus irreducible.
\item
We now compute $H^1(G,\tmop{Pic}Y)$. The kernel of $1+\s$ is generated by $s_1-s_i$, for $i\in\{2,\ldots,n\}$, while the image of $1-\s$ is generated by $s_1-s_2$ and $2(s_1-s_i)$, for $i\in\{3,\ldots,n\}$. Thus $H^1(G,\tmop{Pic}Y)\simeq(\mb{Z}/2\mb{Z})^{n-2}$. Moreover, all relations satisfy the overlap condition by Proposition \ref{ol}. By Remark \ref{inbr}, for $m_i\in\{0,1\}$ not all zero, the corresponding $A\assign\oy\oplus(\otimes_{i=3}^nL_i^{m_i})_{\s}$ are maximal orders on $\p2$ ramified on $C$, distinct in $\tmop{Br}(K(Z))$.
\end{enumerate}
\end{proof}
\begin{figure}[H]
\begin{center}
\includegraphics*[trim= 0 5 0 0]{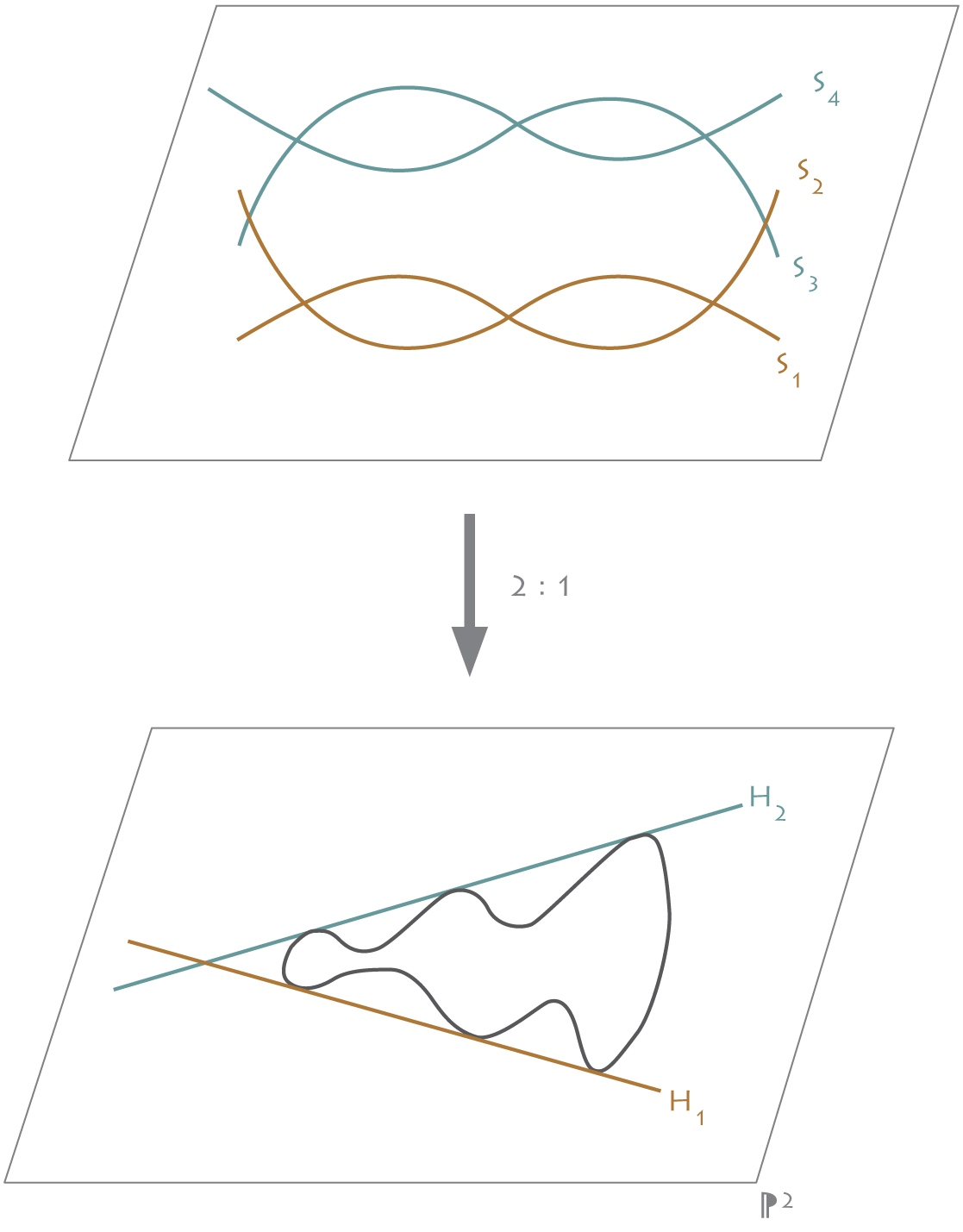}
\end{center}
\caption{The double cover $\pi\colon Y\to\p2$ in the $n=3$ case.}
\end{figure}
\begin{remark}\label{bita}
Since $C$ is smooth and irreducible, by Lemma \ref{inj}, $\pi^*$ is an isomorphism onto $(\tmop{Pic}Y)^G\simeq\mb{Z}(s_1+s_2)$, implying $\pi^*H=s_1+s_2$, where $H$ is a line on $Z$. Then for all $i$,  $S_i+\phi(S_i)\sim S_1+S_2$ is the inverse image of a line on $\p2$. Since $S_i$ and $\phi(S_i)$ are distinct rational curves, there are $(n-1)$ lines $H_i\subset\p2$ such that $\pi^{-1}(H_i)=S_i+\phi(S_i)$.
Noting that for all $i$, $S_i\cdot\phi(S_i)=3$, this then implies that $C$ has $n-1$ tritangents $H_1,\ldots,H_{n-1}$.
\end{remark}

\section{The construction of orders on ruled surfaces}\label{ruled}

We now  construct orders on ruled surfaces, to which end we make the following definition.
\begin{definition}\index{ruled surface}
 A surface $Z$ is {\bf (geometrically) ruled} if there exists a smooth curve $C$ and a morphism $p\colon Z\to C$ such that, for all $c\in C$, the fibre $Z_c\simeq\p1$.
\end{definition}

Given a ruled surface $\rho\colon Z\to C$, we know that $Z=\mb{P}_C(E)$\index{$\mb{P}_C(E)$} the projectivisation of a rank $2$ vector bundle $E$ on $C$ (see \cite[Prop. III.7]{bo}).

\begin{proposition}\label{ruledpic}\cite[Proposition III.18]{bo}
 The Picard group of a ruled surface $Z$ with ruling $p:Z\to C$ is given by
\begin{eqnarray*}
\tmop{Pic}Z & = & p^*\tmop{Pic}C\oplus\mb{Z}C_0.
\end{eqnarray*}
Moreover, $C_0^2=\tmop{deg}(E),F^2=0,C_0\cdot F=1$ and $K_Z\equiv-2C_0-(\tmop{deg}(E)+2g(C)-2)F,$ where $F$ is a fibre of $p.$ 
\end{proposition}

\begin{example}\index{Hirzerbruch surface}
The $n$th Hirzebruch surface, ${\mb{F}_n\assign\mb{P}_{\mb{P}^1}(\os_{\mb{P}^1}\oplus\os_{\mb{P}^1}(-n))}$\index{$\mb{F}_n$}, is ruled over $\p1$. In fact, these are all the ruled surfaces over the projective line (up to isomorphism) \cite[Prop.III.7]{bo}. 
\end{example}
\begin{remark}\label{hirz}

The Picard lattice of $\fn$ is given by $\tmop{Pic}\fn\simeq\mb{Z}^2$ with intersection form
$$
\left( \begin{array}{cc}
-n & 1 \\
1 & 0 
\end{array} \right) 
$$
If $n$ is even, this Picard lattice is isometric to
$$\mb{H}=
\left( \begin{array}{cc}
0 & 1 \\
1 & 0 
\end{array} \right) 
$$
with respect to the generators $C_0+\frac{n}{2}F, F$ of $\tmop{Pic}\fn$. If $n$ is odd, $\pic\fn$ is not isometric to $\mb{H}$ since $C_0^2=-n$ and $\mb{H}$ is an even lattice.
\end{remark}

\begin{lemma}\label{negdiv}
 Let $n>0$. The only Hirzebruch surface with an irreducible divisor $C_0$ such that $C_0^2=-n$ is $\fn$.
\begin{proof}
Assume there exists an irreducible curve $C\subset\mb{F}_m$ such that $C^2=-n$, where $n\neq m$. We know  that there exists a $C_0\subset\mathbb{F}_m$ such that $C_0^2=-m$. Since $C$ is linearly equivalent to neither $C_0$ nor $F$, by \cite[Chap. V, Cor. 2.18]{ag}, $C\sim aC_0+bF$,where $a,b>0$. However, since $C$ is irreducible, $C\cdot C_0\geq 0$ and $C\cdot F\geq 0$, implying $C^2=C\cdot(aC_0+bF)\geq 0$, contradicting the existence of such a curve $C$.
\end{proof}

\end{lemma}

\subsection{Orders on $\p1\times\p1$}
To demonstrate the versatility of the construction introduced in Section \ref{sextic}, we now perform the same trick to construct orders on the quadric surface $\p1\times\p1$ with ramification locus a $(4,4)-$divisor $D$. As before, we begin by constructing a K3 surface $Y$ which we shall eventually show to be a double cover of the quadric.
\begin{proposition}\label{p1xp1}
 Let $S=\mb{Z}^4=\langle s_1,s_2,s_3,s_4\rangle$ with bilinear form given by 
\begin{eqnarray*}
\left( \begin{array}{cccc}
0 & 1 & 1 & 1\\
1 &-2 & 2 & 0\\
1 & 2 &-2 & 0\\
1 & 0 & 0 &-2\\
\end{array} \right)
\end{eqnarray*}
Then there exists a K3 surface $Y$ such that $\tmop{Pic}Y\simeq S.$ Moreover, $s_2, s_3,s_4$ and $s_5=s_2+s_3-s_4$ are effective nodal classes while the general member of $s_1$ is an irreducible curve of arithmetic genus $1$.

\begin{proof}
We embed $S$ in $\Ll$ via
 \begin{equation*}
\begin{aligned}
\gamma\colon s_1 & \mapsto\mu_1+\mu_1', & s_2 & \mapsto\la1+\mu_2+\mu_1'',\\
s_3& \mapsto \la4+\mu_2+\mu_2'', & s_4& \mapsto	\la2+\mu_2.
\end{aligned}
\end{equation*}
Since $\{\gamma(s_1),\ldots,\gamma(s_4),\la1,\ldots,\la8,\la1',\ldots,\la8',\mu_1,\mu_2'\}$ is a basis for $\Ll$, by Remark \ref{prim} $\gamma$ is a primitive embedding. Since $S$ also has signature $(1,3)$ (this calculation was performed by Maple), there exists a K3 surface $Y$ such that $\tmop{Pic}Y\simeq S$ by Proposition \ref{mor}. We may once again assume that the $s_i$ and $s_5\assign s_2+s_3-s_4$ are effective classes. Then explicit elementary computations demonstrate that $s=s_1+s_2+s_3$ satisfies the conditions of Lemma \ref{ample}, implying that $s$ is an ample class. For $i\in\{2,3,4,5\}$, $s\cdot s_i=1$ implying that each $S_i$, for $i\in\{2,3,4,5\}$, is irreducible and thus a nodal curve. 

We now show that the generic member of $|s_1|$ is irreducible:
since $s$ is ample and $s\cdot s_1=2$, any member of $|s_1|$ has at most two components. By \cite[Prop. 2.6]{projk31}, if $|s_1|$ has no fixed components, then every member of $|s_1|$ can be written as a finite sum $E_1+\ldots +E_n$ where $E_i\sim E$ for all $i$ and $E$ an irreducible curve of arithmetic genus 1. Then $s_1\cdot s_2$ is a multiple of $n$. This, along with the fact that  $s_1\cdot s_2=1$, implies that $n=1$. From \cite[Discussion (2.7.3)]{projk31}, $|s_1|$ has fixed components if and only if the generic member of $|s_1|$ is $E\cup R$, where $E$ is an irreducible curve of arithmetic genus 1 and $R$ is both a nodal curve and a fixed component of $|s_1|$. 

We know the following: $E+R\in|s_1|$, $R^2=-2$, $E^2=0$, and $s_1^2=0$; it follows that $E\cdot R=1$ and $R\cdot S_1=-1$. 
Then the intersection theory on $Y$ yields that $R\not\sim S_i$ for $i\in\{2,3,4\}$. 
Then $R\cdot S_4\geq 0$ and $E\cdot S_4\geq 0$. From \cite[discussion preceding Prop. 2.6]{projk31}, the curve $E$ defines a base-point free pencil of genus 1 curves on $Y$ and $S_4$ is not a component of any fibres, implying that $E\cdot S_4=1$ and thus $R\cdot S_4=0$. Similarly we see that $R\cdot S_2=R\cdot S_3=0$.  Letting $R\sim\sum_{i=1}^4a_iS_i$,
\begin{eqnarray}\label{s2}
 R\cdot S_2 & = & a_1-2a_2+2a_3\\ \label{s3}
 R\cdot S_3 & = & a_1+2a_2-2a_3\\ \label{s4}
 R\cdot S_4 & = & a_1-2a_4.
\end{eqnarray}
Since $R\cdot S_i=0$, for $i\in\{2,3,4\}$, (\ref{s2})$+$(\ref{s3}) yields $a_1=0$ and it follows that $a_2=a_3$. This, in conjunction with (\ref{s4}), tells us that $a_4=0$. Then $R\sim a_2(S_2+S_3)$, implying $R^2=0$ and thus $R$ cannot possibly be a nodal curve, yielding a contradiction.
We conclude that $|s_1|$ has no fixed components and its general member is an irreducible curve of arithmetic genus 1.
\end{proof}

\end{proposition}

\begin{proposition}\label{p1}
Let $Y$ be as in Proposition \ref{p1xp1}. Then
\begin{enumerate}[i)]
 \item
there exists an automorphism $\s$ of $Y$ such that $Z\assign Y/G$ is $\p1\times\p1$ (where $G=\langle\s|\s^2=1\rangle$) and $\pi:Y\to Z$ is a double cover ramified on a $(4,4)-$divisor $D$;
\item
${H^1(G,\tmop{Pic}Y)\simeq\mb{Z}/2\mb{Z}}$, generated by $L\simeq s_2-s_4$. Then A$\assign \oy\oplus\ls$ is a maximal order on $Z$ ramified on $D.$
\end{enumerate}
\begin{proof}
\begin{enumerate}[i)]
\item
As before, we first define an involution $\phi$ on $\tmop{Pic}Y\oplus T_Y.$ The action of $\phi$ on $\tmop{Pic}Y$ is given by the matrix
 \begin{eqnarray*}
\left( \begin{array}{cccc}
1 & 0 & 0 & 0\\
0 & 0 & 1 & 1\\
0 & 1 & 0 & 1\\
0 & 0 & 0 & -1\\
\end{array} \right)
\end{eqnarray*}
and $\phi(t)=-t$, for all $t\in T_Y.$ Once again, MATLAB verifies that $\phi$ extends to an isometry on $H^2(Y,\mb{Z})$, also denoted $\phi$. We now show that $\phi$ is an effective Hodge isometry. Since $H^{2,0}(Y)\subset T_Y$, $\phi$ preserves $H^{2,0}(Y)$. The ample class $s=s_1+s_2+s_3$ (ampleness of $s$ was demonstrated in Propostion \ref{p1xp1}) is preserved by $\phi$ and thus $\phi$ is an effective Hodge isometry. We conclude from the Strong Torelli theorem that $\phi$ is induced by a unique involution $\s\colon Y\to Y$. We denote by $\pi\colon Y\to Z\assign Y/G$ the corresponding quotient morphism, where $G=\langle\s|\s^2=1\rangle$. 
Since $\pi^*\colon \tmop{Pic}Z \to (\tmop{Pic}Y)^G=\langle s_1,s_2+s_3\rangle$ is an injection by Lemma \ref{inj}, $\tmop{rk}\tmop{Pic}Z\leq 2$. The K3 double cover of an Enriques surface has Picard rank $\geq 10$, implying by Proposition \ref{k3sym} that $Z$ is rational and $\pi$ is ramified on the disjoint union of smooth curves. 
\\
 We now show that the ramification locus $D'\subset Y$ is irreducible: firstly, any component of $D'$ is necessarily fixed by $\s$ and thus linearly equivalent to $aS_1+b(S_2+S_3)$, implying its self-intersection is $4ab$. Thus there are no components of $D'$ which are nodal curves.  Moreover, since $\tmop{rk}\tmop{Pic}Z\leq 2$, we conclude from Remark \ref{irred} that $D'$ is irreducible.
Then by Lemma \ref{inj}, $\pi^*\colon \tmop{Pic}Z \to (\tmop{Pic}Y)^G\simeq\mb{Z}s_1\oplus\mb{Z}(s_2+s_3)$ is an isomorphism, implying $\tmop{Pic}Z\simeq \langle t_1,t_2\rangle$ (where $\pi^*(t_1)=s_1, \pi^*(t_2)=s_2+s_3$). By Corollary \ref{intform}, $\tmop{Pic}Z$ has intersection form given by
$$\mb{H} = 
\left( \begin{array}{cc}
0 & 1 \\
1 & 0 
\end{array} \right) 
$$
Since the rational surfaces with Picard rank 2 are precisely the Hirzebruch surfaces $\mb{F}_n,$  by Remark \ref{hirz}, $Z\simeq\mb{F}_{2n}$, $n\geq0 $. Assuming $Z\simeq\mb{F}_{2n}, n>0$, there exists an effective divisor $C$ such that $C^2=-2n$. We now show that this is impossible: such a $C$ is necessarily linearly equivalent $at_1+bt_2$, where $ab=-2n$. Thus there exists an effective divisor $C$ linearly equivalent to $aS_1+b(S_2+S_3)$. Since $|s_1|$ defines a base-point free pencil on $Y$, $C\cdot S_1\geq 0$. However, $C\cdot S_1=2b$ and we conclude that $b\geq 0$. Moreover, since $(s_2+s_3)^2=0$, $|s_2+s_3|$ is a pencil and any fixed component is either $S_2$ or $S_3$, which isn't possible since then $S_2$ or $S_3$ would give a pencil of curves with no fixed component by \cite[discussion (2.7)]{projk31}. Thus $|s_2+s_3|$ defines a base-point free pencil and $C\cdot(S_2+S_3)\geq 0$, implying that $a\geq 0$. Thus $ab\geq 0\neq-2n$, yielding a contradiction.
Thus $n=0$ and $Z\simeq\mb{F}_0$, that is, isomorphic to $\p1\times\p1$. Use of the Hurwitz formula once again demonstrates that $\pi$ is ramified on a $(4,4)$-divisor $D$.
\item 
The kernel of $1+\s$ is generated by $s_2-s_3$ and $s_2-s_4$ while the image of $1-\s$ is generated by $s_2-s_3$ and $2(s_2-s_4)$.
Thus ${H^1(G,\tmop{Pic}Y)\simeq\mb{Z}/2\mb{Z}}$, generated by $L=s_2-s_4$. By Proposition \ref{ol}, the nontrivial relation satisfies overlap and by Remark \ref{inbr}, $A\assign\oy\oplus L_\s$ is a nontrivial order  on $Z$ ramified on $D$ and thus maximal.
\end{enumerate}
\end{proof}
\end{proposition}
\begin{remark}\label{bit}
$S_2+\phi(S_2),S_4+\phi(S_4)$ are inverse images of fibres of the same projection $p_2\colon Z\to \p1$. Thus there are two fibres $F,G$ of $p_2$ such that $\pi^{-1}(F)=S_2+\phi(S_2)$ and $\pi^{-1}(G)=S_4+\phi(S_4)$ and $A\simeq\oy\oplus\oy(S_2-S_4)_\s$. Thus $F$ and $G$ are both bitangent to the ramification curve $D$.
 \end{remark}

\begin{figure}[H]
\begin{center}
\includegraphics*{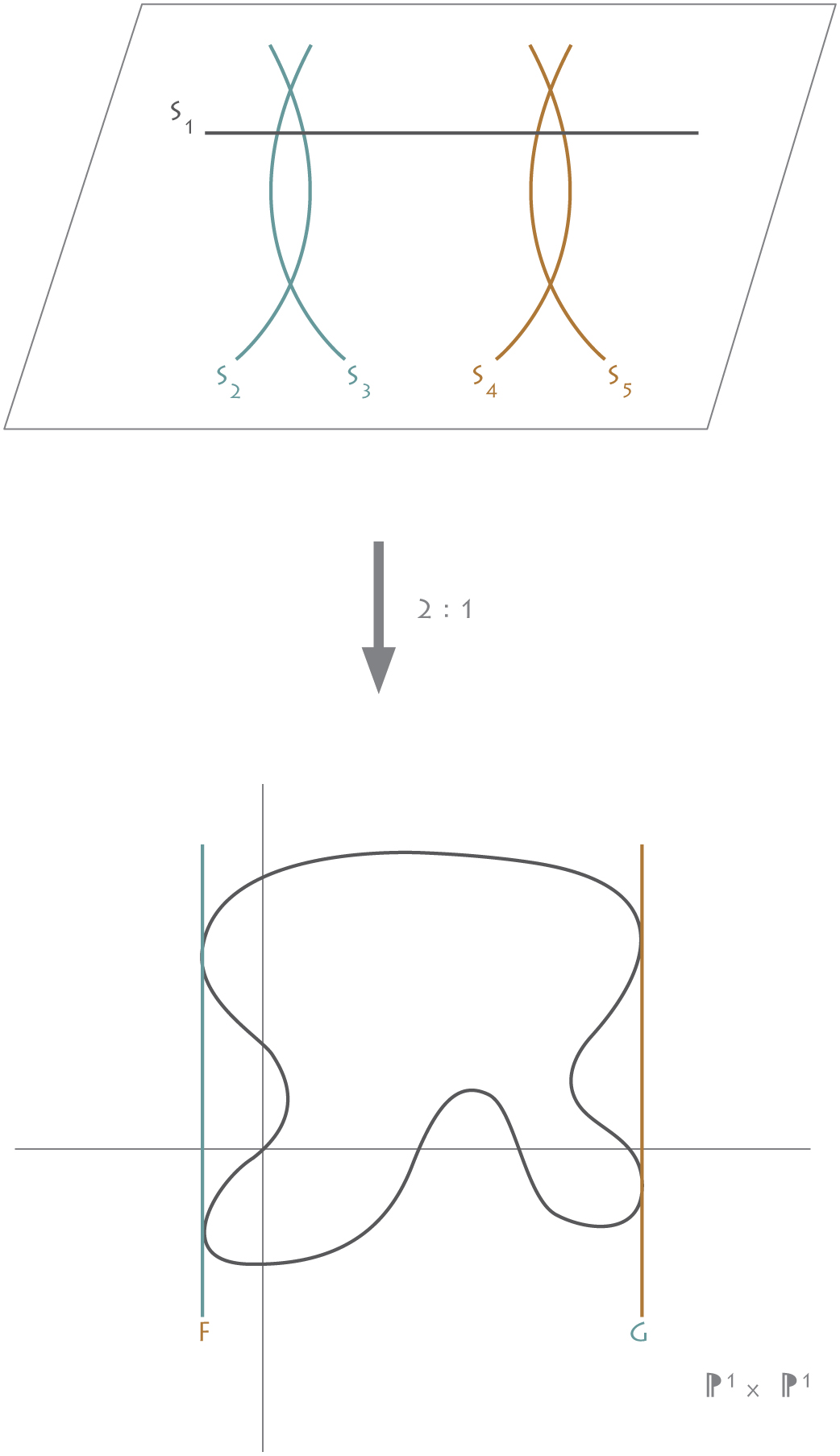}
\end{center}
\caption{The double cover $\pi\colon Y\to\p1\times\p1$.}
\end{figure}
Since the centre of $A$ is the quadric $\p1\times\p1$ and hence birational to $\p2$, it makes sense to ask whether $A$ is itself birational to any of the orders on $\p2$ constructed above.
One may expect $A$ to be birational to an order on $\p2$ ramified on a sextic. This, however, is not the case. In fact, it is birational to an order $A'$ ramified on a singular octic, as we discover below. First let us define what it means for two orders to be birational.

\begin{definition}\index{order!birational equivalence of orders}
 Let $A$ be an order on a scheme $Z$, $A'$ an order on a scheme $Z'$. We say that $A$ and $A'$ are {\bf birational}
 if there exists a birational map $f\colon Z\dashrightarrow Z'$ and an open set $U\subset Z$ on which $f$ is regular such that there is an isomorphism $f_{|U}^*(A')\simeq A_{|U}$.
\end{definition}

This definition will suffice for our purposes. More generally, one may wish to define two such orders to be birational if $f_{|U}^*(A')\sim_M A_{|U}$, where $\sim_M$ denotes Morita equivalence.

Now let $A$ be the order on $Z:=\p1\times\p1$ ramified on a $(4,4)-$divisor $D$ constructed above in Proposition \ref{p1}. We now show that $A$ is birational to an order $A'$ on $Z'\simeq\p2$ ramified on an octic with 2 quadruple points. There exists a birational map $\mu\colon Z \dra Z'$ which is obtained by blowing up a point $p$ on $Z$ (where $p\not\in D, F$ or $G$ ($F,G$ as defined in Remark \ref{bit})) and then blowing down the horizontal and vertical fibres through that point. For any curve $C\subset Z$, we denote its strict transform on $Z'$ by  $C'$. 

The strict transform $D'\subset Z'$ of $D$ is an octic with two 4-fold points $q_1,q_2$ and there is a double cover $\pi'\colon Y'\to Z'$ ramified on $D'$ and a commutative diagram
\begin{displaymath}
 \xymatrix{
Y \ar@{-->}[r]^{\mu_Y} \ar[d]_{\pi} & Y' \ar[d]^{\pi'}\\
Z \ar@{-->}[r]^\mu        & Z'}
\end{displaymath}

\begin{proposition}
 There exists an order $A'$ on $Z'$ ramified on $D'$ which is birational to $A$. Moreover, $A'$ is constructible using the noncommutative cyclic covering trick.
\begin{proof}
 Recall from Remark \ref{bit} that $D$ has two bitangents $F,G$. Then $F',G'$ are both bitangent to $D'$ and thus each splits into two components on $Y'$, say $F_1',F_2'$ and $G_1',G_2'$. Then $A'=\os_{Y'}\oplus\os_{Y'}(F_1'-G_1')$ is a nontrivial order ramified on $D'$. Letting $U\subset Z$ be an open subset of $Z$ such that $\mu_{|U}$ is an isomorphism, we see that $A_{|U}\simeq\mu_{|U}^*(A')$.
\end{proof}

\end{proposition}

\subsection{Orders on $\mb{F}_2$}

Using the same procedure, we now demonstrate ways to construct numerically Calabi-Yau orders on the 2nd Hirzebruch surface $\mathbb{F}_2\assign\mb{P}_{\mb{P}^1}(\os_{\mb{P}^1}\oplus\os_{\mb{P}^1}(-2))$.

\begin{proposition}\label{f2yo}
 Let $S=\langle s_1,s_2,s_3,s_4,s_5\rangle$ with bilinear form given by 
\begin{eqnarray*}
A=\left( \begin{array}{ccccc}
-2 & 0 & 1 & 0 & 1\\
0 & -2 & 0 & 1 & 0\\
1 & 0 & -2 & 2 & 0\\
0 & 1 & 2 & -2 & 0\\
1 & 0 & 0 & 0 & -2
\end{array} \right)
\end{eqnarray*}
Then there exists a K3 surface $Y$ such that $\tmop{Pic}Y\simeq S.$ Moreover, $s_i$ is an effective nodal class for all $i$.
\end{proposition}

\begin{proof}
We embed $S$ in $\Ll$ via
 \begin{equation*}
\begin{aligned}
\gamma\colon s_1 & \mapsto\la4, & s_2 & \mapsto\la2+\mu_1, &s_3& \mapsto \la1+2\mu_1,\\
 s_4& \mapsto\la7+\mu_2, & s_5 & \mapsto \la5.
\end{aligned}
\end{equation*}
Since $\{\gamma(s_1),\ldots,\gamma(s_5),\la3,\la6,\la7,\la8,\la1',\ldots,\la8',\mu_1,\mu_1',\mu_2',\mu_1'',\mu_2''\}$ is a basis for $\Ll$, $S\hookrightarrow\Ll$ is a primitive sublattice by Remark \ref{prim}. Since $S$ also has signature $(1,4)$ (this calculation was performed by Maple), there exists a K3 surface $Y$ such that $\tmop{Pic}Y\simeq S$ by Proposition \ref{mor}. By Proposition \ref{rr}, we may assume that $s_i$ are effective classes, as is $s_6\simeq s_3+s_4-s_5$. A simple calculation verifies that $s=s_1+s_2+3s_3+3s_4$ satisfies the conditions of Lemma \ref{ample}, implying $s$ ample and since $s\cdot s_i=1$ for all $i$, each $s_i$ is an irreducible class and hence an effective nodal class.
\end{proof}
\begin{proposition}\label{f2}
\begin{enumerate}[i)]
Let $Y$ be as in Proposition \ref{f2yo}. Then 
 \item 
there exists an automorphism $\s$ of $Y$ such that $Z\assign Y/G$ is $\mb{F}_2$ (where $G=\langle\s|\s^2=1\rangle$) and $\pi\colon Y\to Z$ is a double cover ramified on a divisor ${D\sim 4C_0+8F }$;
\item
${H^1(G,\tmop{Pic}Y)\simeq\mb{Z}/2\mb{Z}}$, generated by $L=s_3-s_5$ and $A=\oy\oplus\ls$ is a maximal order on $Z$ ramified on $D.$
\end{enumerate}

\begin{proof}
\begin{enumerate}[i)]
\item
As before, we first define an involution $\phi$ on $\tmop{Pic}Y\oplus T_Y.$ The action of $\phi$ on $\tmop{Pic}Y$ is given by the matrix
 \begin{eqnarray*}
\left( \begin{array}{ccccc}
0 & 1 & 0 & 0 & 0\\
1 & 0 & 0 & 0 & 0\\
0 & 0 & 0 & 1 & 1\\
0 & 0 & 1 & 0 & 1\\
0 & 0 & 0 & 0 & -1
\end{array} \right)
\end{eqnarray*}
and $\phi(t)=-t$, for all $t\in T_Y.$ Once again, MATLAB verifies that $\phi$ extends to an isometry on $H^2(Y,\mb{Z})$, also denoted $\phi$. We now show that $\phi$ is an effective Hodge isometry. Since $H^{2,0}(Y)\subset T_Y$, $\phi$ preserves $H^{2,0}(Y)$. Moreover, since $\phi$ preserves the ample class $s=s_1+s_2+3(s_3+s_4)$ (demonstrated to be ample in Proposition \ref{f2yo}), $\phi$ is an effective Hodge isometry. \\
Since $\pi^*\colon \tmop{Pic}Z \to (\tmop{Pic}Y)^G=\langle s_1+s_2,s_3+s_4\rangle$ is an injection by Lemma \ref{inj}, $\tmop{rank}\tmop{Pic}Z\leq 2$. The K3 double cover of an Enriques surface has Picard rank $\geq 10$, implying by Proposition \ref{k3sym} that $Z$ is rational and $\pi$ is ramified on the disjoint union of smooth curves. 
We now show that the ramification $D'$ locus is irreducible: firstly, any component of $D'$ is necessarily fixed by $\s$ and thus linearly equivalent to $a(S_1+S_2)+b(S_3+S_4)$, implying its self-intersection is $4a(b-a)$. Thus there are no components of $D'$ which are nodal curves. Moreover, since $\pi^*\colon \tmop{Pic}Z \to (\tmop{Pic}Y)^G$ is an injection by Lemma \ref{inj}, $\tmop{rank}\tmop{Pic}Z\leq 2$. We conclude from Remark \ref{irred} that $D'$ is irreducible.

Thus by Corollary \ref{intform}, $\tmop{Pic}Z\simeq(\tmop{Pic}Y)^G\simeq\langle s_1+s_2,s_3+s_4\rangle$ with intersection product given by
$$
\left( \begin{array}{cc}
-2 & 1 \\
1 & 0 
\end{array} \right) 
$$
Thus $Z$ is a Hirzebruch surface $\mathbb{F}_{2n}$, $n\geq 0$. Letting ${T_1=\pi(S_1)}$, $T_1$ is an irreducible divisor such that $T_1^2=-2$, implying by Lemma \ref{negdiv} that  $Z\simeq\mb{F}_2$. Since $Y$ is a K3 surface, $\pi$ is ramified on a divisor $D\sim -2K_{\mb{F}_2}$ and thus $D\sim 4C_0+8F$. 
\item
We now compute $H^1(G,\tmop{Pic}Y)$: the kernel of $1+\s$ is generated by ${s_1-s_2},{s_3-s_4}$ and $s_3-s_5$ while the image of $1-\s$ is generated by $s_1-s_2,s_3-s_4$ and $2(s_3-s_5)$. Thus $H^1(G,\tmop{Pic}Y)\simeq\mb{Z}/2\mb{Z}$, generated by $s_3-s_5$.
The construction of the maximal order follows as in previous examples: $A=\oy\oplus\ls$, where $L$ is a representative of $s_3-s_5$.
\end{enumerate}
\begin{remark}
$S_3+\phi(S_3),S_5+\phi(S_5)$ are inverse images of fibres of the projection $p\colon\mb{F}_2\to \p1$. Thus there are two fibres $F,G$ of $p_2$ such that $\pi^{-1}(F)=S_3+\phi(S_3)$ and $\pi^{-1}(G)=S_5+\phi(S_5)$ and $A\simeq\oy\oplus\oy(S_3-S_5)_\s$. Thus $F$ and $G$ are both bitangent to the ramification curve $D$.
 \end{remark}

\begin{figure}[H]
\begin{center}
\includegraphics*[trim= 0 0 0 0]{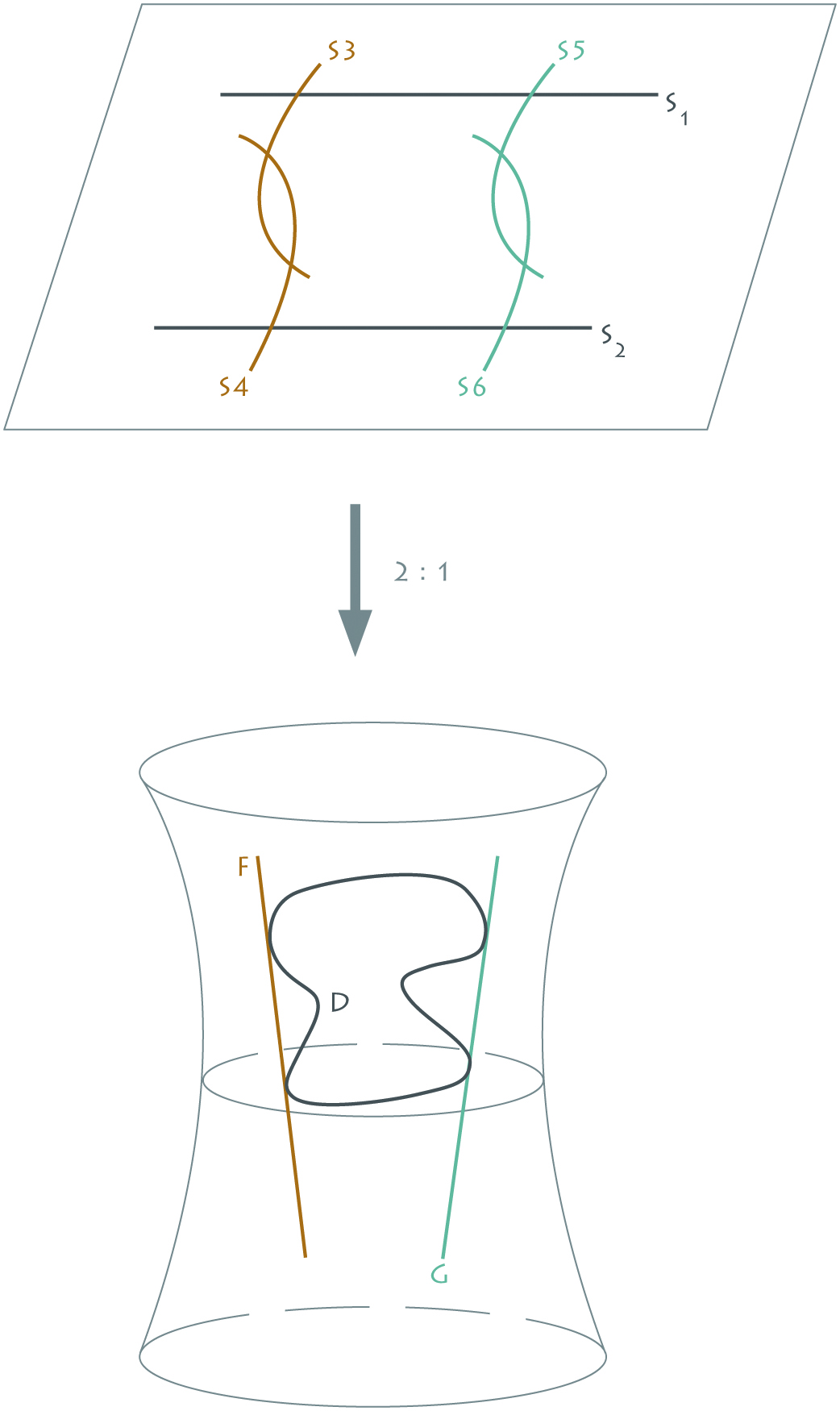}
\end{center}
\caption{The double cover $\pi\colon Y\to\mathbb{F}_2$.}
\end{figure}

\end{proof}
\end{proposition}

\subsection{When the surface is ruled over an elliptic curve}\label{ellruled}
Up until now, our use of the Torelli theorem has required that we work over $\mb{C}$. This is no longer required and as such, for the remainder of this chapter, the base field is an arbitrary algebraically closed field $k$ of characterstic $0$.
We now include a brief section on the construction of numerically Calabi-Yau orders on surfaces ruled over elliptic curves.
These orders are of particular interest since there are such limited possibilities for ramification data. Indeed, for this reason Chan and Kulkarni remark that they "are inclined to think that these orders are somehow special"\cite{ncy}.

Let $A$ be a numerically Calabi-Yau order on a surface $Z$ which is ruled over an elliptic curve:
\begin{eqnarray*}
p: Z & \lrw C
\end{eqnarray*}
We know that $Z=\mb{P}_C(E),$ the projectivisation of a rank $2$ vector bundle $E$ on $C$. Since $Z$ is the centre of a numerically Calabi-Yau order, the number of permissible $E$'s is finite and easy to enumerate \cite{ncy}: 
\begin{enumerate}
\item if $E$ is split, then $E=\os_C\oplus N,$ where $N$ is $a-$torsion for $a\in\{1,2,3,4\}.$ 
\item If $E$ is not split, it is indecomposable of degree one, that is, the non-split extension of a degree one line bundle $L$ by $\os_C:$
\begin{eqnarray*}
0\lrw \os_C\lrw E\lrw L \lrw 0
\end{eqnarray*}
\end{enumerate}

\subsubsection{Case 1: $E$ splits}
We would like to describe the possible ramification of such orders and for this reason we first wish to describe the Picard group of $Z$: recall from Proposition \ref{ruledpic} that $\tmop{Pic}Z = p^*\tmop{Pic}C\oplus\mb{Z}C_0$ and $C_0^2=\tmop{deg}(E),F^2=0,C_0\cdot F=1$ and $K_Z\equiv-2C_0+\tmop{deg}(E)F$.

\begin{remark}
 In our case, $E=\os_C\oplus N,$ where $N$ is $a-$torsion. This implies that ${\tmop{deg}(E)=0}, C_0^2=0$ and $K_Z\sim-2C_0$.
\end{remark}

We see in \cite[discussion following Lemma 2.4]{ncy} that we can record ramification data using ramification vectors: $(e_1,e_2,\ldots)$ where the $e_i$ are repeated with multiplicity. In the case of an order on a ruled surface, given ramification curves $D_i$, the multiplicity is of the form $D_i\cdot F,$ where $F$ is a fibre of the ruling. In the following, by {\bf $n-$section}\index{$n-$section} we shall mean an effective divisor $D$ on $Z$ (not necessarily irreducible) which does not contain fibres of $p$ as components and such that $D\cdot F=n$ for all fibres $F$ of $p:Z\lrw C.$
From \cite[Prop. 5.4]{ncy}, the possible ramification vectors in the $E$ is split case are $(2,2,2,2),(3,3,3),(2,4,4) \ \tmop{and} \ (2,3,6).$ We shall be dealing with the first case, that of $(2,2,2,2),$ until mentioned otherwise.

We wish to construct examples of orders ramified on $4-$sections of $p,$ which are disjoint unions of elliptic curves by \cite[Prop. 4.2]{ncy}. To this end, we let $D$ be such a $4-$section. In order to use Chan's noncommutative cyclic covering trick to construct such an order, we would like to have a double cover of $Z$ ramified on $D$:
\begin{proposition}\label{cover}
Let $D$ be a $4-$section of $p.$ Then there exists a surface $Y$ and a $2:1$ cover $\pi:Y\lrw Z$ ramified on $D.$
\begin{proof}
From the proof \cite[Prop. 4.3]{ncy}, we see that $D$ is numerically equivalent to $4C_0.$ This means that $D\sim 4C_0+p^*(L'),$ where $L'\in\tmop{Pic}^0C.$ Since the $2-$multiplication map $[2]:C\lrw C$ is surjective, any element of $\tmop{Pic}^0C$ is $2-$divisible. Hence $D\in\tmop{Pic}Y$ is $2-$divisible and by the cyclic covering trick, there exists a surface $Y$ and a $2:1$ cover $\pi:Y\lrw Z$ ramified on $D$.
\end{proof}
\end{proposition}

Letting $G=\langle\s|\s^2=1\rangle$ be the Galois group of $\pi,$ the following proposition tells us some of the line bundles $L\in\tmop{Pic}Y$ which satisfy $L_{\s}^{\otimes 2}\simeq\oy$ and thus give us noncommutative cyclic covers. We shall require the following lemma:

\begin{proposition}\label{tors}
Let $M\in (\tmop{Pic}C)_2$. Then $L\assign (p\pi)^*M\in\tmop{ker}(1+\s).$ Moreover, unless $D$ is irreducible such the the corresponding cover $D\to C$ has Galois group $\mb{Z}/2\mb{Z}\oplus\mb{Z}/2\mb{Z}$, there exists an $M\in (\tmop{Pic}C)_2$ such that ${A_M=\oy\oplus\ls}$ is a maximal numerically Calabi-Yau order on $Z$ with ramification vector $(2,2,2,2)$.
\begin{proof}
First we note that $(p\pi)^*\tmop{Pic}C$ is $G-$invariant since any element of $\pi^*\tmop{Pic}Z$ is $G-$invariant. Thus $L+\s(L)=2L\simeq\oy\in\tmop{Pic}Y$ and $L\in\tmop{ker}(1+\s).$ By Proposition \ref{ol}, all relations arising from elements of $H^1(G,\tmop{Pic}Y)$ satisfy the overlap condition. We now show that there exists at least one 2-torsion line bundle $M$ such that $A_M=\oy\oplus L_{\s}$ is maximal. We do so by showing that $A_M$ has nontrivial ramification. Let $D_i$ be the curves where $\pi$ (and thus $A_M$) is ramified, that is, the $D_i$ are the components of $D$. Since each $D_i$ is an $n$-section for some $n$, we have a corresponding finite map $\pi_i:D_i\to C.$ Theorem \ref{ram} states that the ramification along $\pi(D_i)$ is isomorphic to the cover given by the $2-$torsion line bundle $L_{|D_i},$ which is exactly $\pi_i^*(M).$ 
Thus $A_M$ has nontrivial ramification on all $D_i$ unless $\pi_i^*(M)\simeq\os_{D_i}$ for some ramification curve $D_i$.

If $D_i$ is a section or a trisection, then all nontrivial $2-$torsion line bundles on $C$ pull back nontrivially. If $D$ is the disjoint union of two bisections $D_1$ and $D_2$ and for each $i$, $M_i$ the unique nontrivial line bundle on $C$ such that $\pi_i^*(M_i)\simeq \os_{D_i},$ then $\pi_i^*(M_1\otimes M_2)$ is nontrivial if $D_1$ and $D_2$ are nonisomorphic. If $D_1\simeq D_2,$ then $\pi_i^*(M')$ is nontrivial for all $M'\not\simeq M_1.$ In the case that $D=D_1$ is an irreducible $4-$section, $\pi_1:D\to C$ is cyclic, corresponding to a $4$-torsion $N\in\tmop{Pic}C$.

Then any $M'\in(\tmop{Pic}C)_2$ such that $M'$ is not a power of $N$ pulls back nontrivially to $D$. By Lemma \ref{max}, $A_M$ is maximal and we are done.
\end{proof}
\end{proposition}

The same trick can be performed for the other cases listed above: $(3,3,3),(2,4,4),$ and $(2,3,6)$. We give the flavour of this by doing so for $(3,3,3)$ and $(2,3,6)$, beginning with the former:

\begin{proposition}
 Let $T$ be a trisection of $p\colon Z\to C$. Then there exists a triple cover $\pi\colon Y\to Z$ totally ramified on $T$. Moreover, there exists an $M\in(\tmop{Pic}C)_3$ such that $A_M\assign\oy\oplus L_\s$ is a maximal numerically Calabi-Yau order on $Z$ ramified on $T$, where $L=(p\pi)^*(M)$.
\begin{proof}
From the proof of \cite[Prop. 4.3]{cyc}, $T\simeq 3C_0+p^*(L')$, for $L'\in\tmop{Pic}^0C$. Since the 3-multiplication map $[3]\colon C\to C$ is surjective, any element of $\tmop{Pic}^0C$ is 3-divisible and thus there exists a triple cover of $\pi\colon Y\to Z$ totally ramified on $T$ and unramified elsewhere. The covering group is $G=\langle\s|\s^3=1\rangle.$ Now for any $M\in(\tmop{Pic}Y)_3$, $L\assign(p\pi)^*M\in\tmop{ker}(1+\s+\s^2)$ and by Proposition \ref{ol}, all relations arising from elements of $H^1(G,\tmop{Pic}Y)$ satisfy the overlap condition. We now show that for any $T$, there exists at least one $M\in(\tmop{Pic}C)_3$ such that $A_M\assign\oy\oplus L_\s$ is a maximal numerically Calabi-Yau order. We do this by showing $A_M$ has nontrivial ramification. From Theorem \ref{ram}, the ramification of $A_M$ along each component $T_i$ of $T$ is $p_i^*(M)$, where $p_i=p_{|T_i}$. We know that $(\tmop{Pic}C)_3\simeq(\mb{Z}/3\mb{Z})^2$, generated by $N_1,N_2$. If $T=T_1$ is irreducible, then $p_1^*(N_j)\sim 0$ for only one of the $N_i$ and if $T$ is reducible, then $N_j\not\in\tmop{ker}p_i^*$, for $j\in\{1,2\}$. Thus by Lemma \ref{max} there always exists an $M\in(\tmop{Pic}C)_3$ such that $A_M$ is maximal.
\end{proof}

\end{proposition}

If $A$ is a numerically Calabi Yau order with ramification vector $(2,3,6)$, then $Z\simeq C\times\p1$ by \cite[Prop. 4.3]{ncy}. We now construct such an order.

\begin{proposition}
 Let  $Z\simeq C\times\p1$ and $p_i\in\p1, i\in\{1,2,3\}$ be three points on the line. Then there exists a $6\colon 1$ cyclic cover $\pi\colon Y\to Z$ ramified with index 2 over $Z_{p_1}$ index $3$ over $Z_{p_2}$ and index 6 over $Z_{p_3}$. Then $H^1(G,\tmop{Pic}Y)\simeq (\mb{Z}/6\mb{Z})^2$, and for any $L\in\hh$ which is neither 2-torsion nor 3-torsion, $A\assign\oy\oplus \ls\oplus\ldots\oplus \ls^5$ is a maximal numerically Calabi-Yau orders with ramification vector $(2,3,6)$.
\begin{proof}
Letting $E$ be the elliptic curve with $j(E)=0$, there exists a $6:1$ cyclic cover $\pi_E\colon E\to\p1$ ramified at $p_1$ with index 2, $p_2$ with index 3 and $p_3$ with index 6 (see the statement and proof of \cite[Chap. III, theorem 10.1]{silverman} for details). We form the fibred product
 \begin{displaymath}
\xymatrix{
Y \ar[r] \ar[d] & C\times\p1 \ar[d]\\
E \ar[r]^{\pi_E} &        \p1
}
\end{displaymath}
Then $Y=C\times E$ and $H^1(G,\tmop{Pic}Y)\simeq H^1(G,\tmop{Pic}C)\oplus H^1(G,\tmop{Pic}E)$.
Since all elements of $\tmop{Pic}C$ are fixed by $G$, $H^1(G,\tmop{Pic}C)\simeq(\tmop{Pic}C)_6\simeq(\mb{Z}/6\mb{Z})^2$. We now show that $H^1(G,\tmop{Pic}E)\simeq 0$: since $H^1(G,\tmop{Pic}E)\simeq\tmop{ker}(1+\s+\ldots+\s^5)/\tmop{im}(1-\s)$ and $\tmop{ker}(1+\s+\ldots+\s^5)\subset\tmop{Pic}^0E$, it suffices to show that $\tmop{im}(1-\s)=\tmop{Pic}^0E$. In fact, we now show that $(1-\s)\colon\tmop{Pic}^1E\to\tmop{Pic}^0E$ is surjective. The group homomorphism $1-\s\colon\tmop{Pic}^1E\to\tmop{Pic}^0E$ corresponds to a scheme morphism on the relevant components of the Picard scheme of $E$,  $\mathbf{Pic}^1E$ and $\mathbf{Pic}^0E$, which are both isomorphic to $E$. Since $\pi$ is totally ramified at $p_3\in\p1$, $q=\pi^{-1}(p_3)$ is fixed by $\s$ and $(1-\s)(q)=e_0$, the zero point of $E$. Letting $q'\in E$ be any point not fixed by $\s$, $(1-\s)(q')\neq e_0$. Then $(1-\s)\colon\mathbf{Pic}^1E\to\mathbf{Pic}^0E$ is non-constant implying it is surjective  by \cite[Chap. II, Prop. 6.8]{ag}.

Thus $H^1(G,\tmop{Pic}E)\simeq 0$ and it follows that $H^1(G,\tmop{Pic}Y)\simeq(\tmop{Pic}C)_6\simeq(\mb{Z}/6\mb{Z})^2$.
 By Proposition \ref{ol}, all relations satisfy the overlap condition. Letting $L=p_1^*(M)$, for $M\in(\tmop{Pic}C)_6$, the ramification vector of $A\assign\oy\oplus \ls\oplus\ldots\oplus \ls^5$ is $(2,3,6)$ and, by Lemma \ref{untot}, is given by $M^3$ over $Z_{p_1}$, $M^2$ over $Z_{p_2}$, and $M$ over $Z_{p_3}$.
By Lemma \ref{max}, $A$ is maximal precisely when $M$ is neither $2$-torsion nor $3$-torsion and the result follows.
\end{proof}
\end{proposition}

We now look at the case when $Z$ arises from an indecomposable vector bundle $E$ on $C$.

\subsubsection{Case 2: $E$ indecomposable}
Recall that in this case $Z=\mb{P}_C(E),$ where $E$ is the non-split extension of a degree one line bundle $L$ by $\os_C:$
\begin{eqnarray*}
0\lrw \os_C\lrw E\lrw L \lrw 0
\end{eqnarray*}
Here
\begin{eqnarray*}
\tmop{Pic}Z & = & p^*\tmop{Pic}C\oplus\mb{Z}C_0,
\end{eqnarray*}
$C_0^2=1$ and $K_Z\equiv-2C_0+F$ (from Proposition \ref{ruledpic}).

\begin{theorem}[\cite{ncy}, Theorem 5.6]\label{enum}
Let $A$ be a numerically Calabi-Yau order on $Z.$ Then the ramification indices are all $2$ and either
\begin{enumerate}
\item the ramification locus $D=D_1$ is irreducible and $D_1\equiv-2K$, or
\item the ramification locus $D=D_1\cup D_2$ splits such that $D_i\equiv-K.$
\end{enumerate}
\end{theorem}
\begin{remark}
In case 1, the divisor is an irreducible $4-$section and in case 2, the disjoint union of 2 irreducible bisections \cite[proof of Thm 4.5]{ncy}.
\end{remark}
\begin{proposition}
Let $D$ be a divisor of the form specified in Theorem \ref{enum}. Then we can construct a maximal order $A=\oy\oplus\ls$ on $Z$ with nontrivial ramification on each component of $D$ unless $D$ is an irreducible $4-$section such that the covering $D\to C$ is not cyclic. Here $\pi\colon Y\to Z$ is the double cover ramified on $D$ and $L$ is the pullback to $Y$ of a $2$-torsion line bundle on the base curve $C$.
\begin{proof}
Let $D$ be an effective divisor numerically equivalent to $-2K.$ Then there exists a surface $Y$ and a $2:1$ cover $\pi:Y\lrw Z$ ramified on $D$ with Galois group $\{\s|\s^2=1\}$. The proof of this is analogous to that of Proposition \ref{cover}.
As in Proposition \ref{tors} above,
\begin{eqnarray*}
(p\pi)^*(\tmop{Pic}C)_2 & \subset & \tmop{ker}(1+\s)
\end{eqnarray*}
and all relations satisfy the overlap condition by Proposition \ref{ol}. As in Proposition \ref{tors} letting $p_i:D_i\lrw C$ denote the projections of each component of $D,$ the ramification of $A=\oy\oplus((p\pi)^*M)_{\s}$ over $D_i$ is given by $p_i^*(M),$ for $M\in(\tmop{Pic}C)_2.$ Now we need to verify that there exists a $2-$torsion line bundle $M$ such that $p_i^*(M)$ is non-trivial in $\tmop{Pic}D_i$, for all $i$. We demonstrate this for each of the possible cases. If $D=D_1$ is irreducible, then  $p_1:D_1\lrw C$ is a $4:1$ cyclic cover. Thus the kernel of
\begin{eqnarray*}
p_1^*: \tmop{Pic}C & \lrw \tmop{Pic}D_1
\end{eqnarray*}
is generated by some $4-$torsion $M_1\in\tmop{Pic}C.$ Then $p_1^*$ trivialises only one of the three nontrivial $2-$torsion line bundles on $C.$ If, on the other hand, $D=D_1\cup D_2$ is the disjoint union of two bisections, then from the proof of \cite[Theorem 4.5]{ncy}, $D_1$ and $D_2$ are two non-isomorphic double covers of $C.$ Letting $L_i\in\tmop{Pic}C$ be the line bundle trivialised by $p_i^*\colon\tmop{Pic}C\to\tmop{Pic}D_i,$ $p_i^*(L_1+L_2)$ is nontrivial in $\tmop{Pic}D_i$ for each $i$, implying maximality by Lemma \ref{max}.
\end{proof}
\end{proposition}

\section{Orders on rational elliptic fibrations}\label{rat_ell}

The classical method of constructing rational elliptic fibrations is as follows: let $C_1,C_2$ be two elliptic curves on $\p2$ intersecting in 9 distinct points $p_1,\ldots,p_9.$ Letting $Z$ be the blowup of $\p2$ at the $9$ $p_i$'s, we then have an elliptic fibration $\varphi:Z\lrw C$ with $C\simeq\p1.$ This is a Jacobian elliptic fibration, that is, one with a section $S_0$, which from here on in will denote the zero section in the group of sections $\Phi$. We wish to construct orders on $Z$ ramified on $C_1\cup C_2,$ where we have abused notation by writing $C_i\subset Z$ for the strict transform of $C_i\subset\p2$. These orders are examples of minimal orders (see \cite[Example 3.4]{ncy}) on non-minimal surfaces. We first construct the double cover of $\pi:Y\lrw Z$ ramified on $C_1$ and $C_2$:
\begin{lemma}\label{y}
There exists a K3 surface $Y$ and a double cover $\pi:Y\lrw Z$ ramified on $C_1\cup C_2$ with ramification index $2.$ Moreover, there is a Jacobian elliptic fibration $\phi_Y\colon Y \to \p1$. 
\begin{proof}
Firstly, we note that for $i\in\{1,2\}$, $C_i$ is the fibre above a point $p_i\in C$. Since $\os_C(p_1+p_2)$ is $2$-divisible in $\tmop{Pic}C$, there is a double cover $\pi_C\colon C'\to C$ ramified solely at $p_1$ and $p_2$. By the Riemann-Hurwitz formula, $C'\simeq\p1$. We form the fibred product
\begin{eqnarray}
  Y & \longrightarrowlim^{\pi} & Z \nonumber\\
  \phi_Y \downarrow & \square & \downarrow \phi \nonumber\\
  C' & \longrightarrowlim^{\pi_C} & C \nonumber
\end{eqnarray}
and it follows that $\pi\colon Y\to Z$ is the double cover of $Z$ ramified solely on $C_1\cup C_2$. Now $Y$ is a resolution of a double sextic $Y'\to\p2$, the sextic possessing only simple singularities. By \cite[ Sect. 1, discussion preceding Proposition A]{doub}, $Y$ is a K3 surface. Moreover, by construction, $Y$ is elliptically fibred over $C'$. We now show that the pullback under $\pi$ of any section of $\phi$ is a section of $\phi_Y$, implying $\phi_Y$ is a Jacobian elliptic fibration. To see this, let $S$ be a section of $\phi$, $F$ a fibre. Then $\pi^*S\cdot\pi^*F=2$. However, recalling $C\simeq\p1$, we see that $\pi^*F\sim 2F'$, where $F'$ is a fibre of $\phi_Y$, implying $\pi^*S\cdot F'=1$.

\end{proof}
\end{lemma}

In order to construct orders ramified on $C_1\cup C_2,$ we require the existence of nontrivial elements of $H^1(G,\tmop{Pic}Y).$ In general, however, it is extremely difficult to even compute $\tmop{Pic}Y.$ There are two special cases in which we can find nontrivial $1-$cocycles and thus construct nontrivial orders.

\subsection{When $\Phi$ has 2-torsion}

The first is when $\Phi$, the group of sections of $\phi\colon Z\to C$, has $2-$torsion. There are many examples of such rational elliptic fibrations and we refer the reader to \cite[Section ``Torsion groups of rational elliptic surfaces'']{rat} for explicit examples.  

\begin{proposition}
If $\Phi$ has a nontrivial $2-$torsion section $S$ such that $S_{|C_i}\nsim S_{0|C_i}$, $i\in\{1,2\}$, then there exists a nontrivial $L\in H^1(G,\tmop{Pic}Y)$ and the corresponding relation satisfies the overlap condition. The cyclic algebra $A=\oy\oplus\ls$ is a numerically Calabi-Yau order ramified on $C_1\cup C_2$.

\begin{proof}
Since $S$ is a $2$-torsion section, $2(S-S_0)\sim \alpha F,$ where $F$ is the class of a fibre. Let $S'=\pi^*S, S_0'=\pi^*S_0,$ and $F'$ be the class of a fibre of $\phi_Y\colon Y\to C$. Then $\pi^*F=2F'$ and
\begin{eqnarray*}
\s^*: \tmop{Pic}Y & \lrw  & \tmop{Pic}Y\\
       S'         & \longmapsto &  S'\\
       S_0'          & \longmapsto &  S_0'\\
       F'          & \longmapsto & F'
\end{eqnarray*}
and 
\begin{eqnarray*}
S'-S_0'-\alpha F'+\s^*(S'-S_0'-\alpha F') & \sim & 2(S'-S_0')-2\alpha F'\\
                              & \sim & \pi^*(2(S-S_0)-\alpha F)\\
			      & \sim & 0   
\end{eqnarray*}
Thus $L=\oy(S'-S_0'-\alpha F')\in H^1(G,\tmop{Pic}Y).$ Moreover, $L$ is nontrivial in $H^1(G,\tmop{Pic}Y)$ since the ramification of $A=\oy\oplus L_\s$ above $C_i$ is given by the $2$-torsion line bundle $\os_{C_i}(S'_{|Y_{c_i'}}-S'_{0|Y_{c_i'}})$, which is nontrivial by assumption. Thus each $\ti{C_i}$ is irreducible and $A$ is maximal by Lemma \ref{max}.
\end{proof}
\end{proposition}
\begin{remark}
Torsion sections of elliptic fibrations seem to provide a substantial number of examples here. We can perform the same trick for elliptically fibred K3 surfaces: such sections are abundant and have been studied extensively by Persson and Miranda \cite{k3tors}.
\end{remark}

\subsection{The second case in which we can find nontrivial 1-cocycles}

The second case in which we may construct such orders is more subtle and involved, requiring us to study the geometry of a general rational elliptic fibration $\phi:Z\to\p1$ further.

First notice that there is an involution $\tau$ of $Z$ which sends $z\mapsto-z$ on fibres (for details, see \cite[Introduction, p.3]{rat}). Setting $G=\langle\tau|\tau^2=1\rangle$, we let ${\psi:Z\to\ti{X}\assign Z/G}$ be the corresponding quotient morphism and note that $\ti{X}$ is smooth. Blowing down exceptional curves disjoint from $\psi(S_0)$ yields a birational morphism $\mu\colon\ti{X}\to X$ and $X$ is ruled over $C\simeq\p1$. This is because  the quotient of an elliptic curve by the $\mb{Z}_2$-action $z\mapsto -z$ is $\p1$.
\begin{displaymath}
\xymatrix{
	Z  \ar@/^2pc/[rr]^{\rho}\ar[drr]_\phi \ar[r]^\psi & \ti{X} \ar[dr] \ar[r]^\mu & X \ar[d]^p\\
		   &  & \p1}
\end{displaymath}
Letting $\rho\colon Z\to X$ denote the composite morphism $\mu\circ\psi$, $S_0^2=-1$ implies that $\rho(S_0)^2=-2$. Thus $X\simeq\mb{F}_2$ and $\psi$ is ramified on $\ti{C_0}\cup\ti{T}$, where $\ti{C_0}$ is the strict transform of the special section $C_0$ on $\mb{F}_2$ and $\ti{T}$ the strict transform of a trisection $T$ disjoint from $C_0$. Thus we see that classifying rational elliptic fibrations is equivalent to classifying trisections $T$ on $\mb{F}_2$ disjoint from $C_0$ with at most simple singularities and such that $\ti{C_0}+\ti{T}$ is $2$-divisible in $\tmop{Pic}(\ti{X})$. For any effective divisor $D$ on $X$, we let $\ti{D}$ denote its strict transform on $\ti{X}$ with respect to the birational morphism $\mu$.

\begin{theorem}\label{ncyrat}
If $\phi:Z\lrw\p1$ corresponds to a trisection $T\subset\mathbb{F}_2$ with 2 nodes, then we can construct a numerically Calabi-Yau order $A$ ramified on 2 fibres of $\phi.$ The explicit construction of $A$ is given in the proof of the theorem.
\end{theorem}
Before we prove this theorem we require the following lemma.

 \begin{lemma}\label{lemlem}
Assume there exists a section $S$ of $p\colon X\to\p1$ such that $S\cdot C_0=0$ and ${S'=\psi^{-1}(\ti{S})}$ is irreducible. Then there exists an $L\in\tmop{Pic}Y$ (where $Y=Z\times_CS'$) such that $A=\oy\oplus L_\s$ is an order on $Z$ ramified on $Z_{c_i}$, where the $c_i$ are the ramification points of $\pi_C\colon S'\to C$.
\begin{proof}
We first note that since $S$ is a section of $p\colon X\to C$, $S'=\rho^{-1}(S)$ is a bisection of $\phi:Z\to C$ and there is a corresponding irreducible double cover $\pi_C:S'\lrw C$.
We form
\begin{eqnarray}
  Y & \longrightarrowlim^{\pi} & Z \nonumber\\
  \varphi' \downarrow & \square & \downarrow \varphi \nonumber\\
  S' & \longrightarrowlim^{\pi_C} & C \nonumber
\end{eqnarray}
Now $\pi^{-1}(S')=S'\times_CS'$ splits into $2$ copies of $S',$ say, $S_1$ and $S_2.$ Recalling that
\begin{eqnarray*}
 \tmop{Pic}X & \simeq \mb{Z}C_0\oplus\mb{Z}F
\end{eqnarray*}
$S\sim C_0+nF\in\tmop{Pic}X$, implying 
\begin{eqnarray}\label{one}
S' & \assign & \psi^{-1}(\ti{S})\sim2S_0+nF',
\end{eqnarray}
where $F'$ is the class of a fibre of $\phi.$ Let $S_0'=\pi^{-1}S_0$, the reduced inverse image of $S_0$. Then $\s(S_0')=S_0'$ , $\s(S_1)=S_2$ and $\s(F')=F'$.

Then
\begin{eqnarray*}
 (1+\s)(S_1-S_0'-nF')& \sim & S_1+S_2-2S_0'-2nF'\\
                   &  \sim & \pi^*(S'-2S_0-nF)\\
                   &  \sim &   0
\end{eqnarray*}
The last linear equivalence is a result of (\ref{one}). Thus $L\simeq\oy(S_1-S_0+nF') \in H^1(G,\tmop{Pic}Y)$. The corresponding relation satisfies the overlap condition due to Lemma \ref{ol}. By assumption $S\cdot C_0=0$ and this implies $S_1\cdot S_0=0$. We conclude that ${(S_1-S_0)_{|Y_{c_i'}}\not\simeq\os_{Y_{c_i'}}}$. By theorem \ref{ram}, the ramification of $A_S=\oy\oplus L_{\s}$ along $Z_{\pi(c_i')}$ is the cyclic cover corresponding to $\os_{Y_{c_i'}}(S_{1|Y_{c_i'}}-S_{0|Z_{c_i'}})$, which is nontrivial. By Lemma \ref{max}, $A$ is maximal.
\end{proof}
\end{lemma}
\begin{proof}(of theorem \ref{ncyrat})
To prove the theorem, first note that if $S'$ in Lemma \ref{lemlem} is rational, then $\pi_C:S'\lrw C$ is ramified above 2 points and the order $A_S$ is ramified on 2 fibres $Z_{c_i}$ and thus numerically Calabi-Yau (this is precisely the example Chan and Kulkarni deduce exists via the Artin-Mumford sequence \cite[Example 3.4]{ncy}). We are thus required to show that if $\phi:Z\lrw C$ corresponds to a trisection with 2 nodes, there exists a section $S$ of $p\colon X\to C$ such that $S'\assign\rho^{-1}(S)$ is an irreducible rational bisection of $\phi\colon Z\to C$.
This is equivalent to showing there exists a section $S$ such that $\ti{S}\cdot(\ti{T}+\ti{C_0})=2$ since then ${S'\lrw S}$ is ramified above 2 points, implying $S'$ rational.
Now since $T$ is a trisection of $\phi$ disjoint from $C_0$, the intersection theory on $\mb{F}_2$ implies that $T\sim3C_0+6F$. We assume $T$ is a $2-$nodal trisection with nodes at $q_1,q_2.$ Since $h^0(\os_X(aC_0+bF))=1-a^2+ab+b$,  $h^0(\os_X(C_0+2F))=4$ and there exists a divisor $S\sim C_0+2F$ through each $q_i$ with direction different to $T$. Now $S\cdot(T+C_0)=6$ implies that $\ti{S}\cdot(\ti{T}+\ti{C_0})=2$ and we are done.
\end{proof}
\begin{figure}[H]\label{tnode}
\begin{center}
\includegraphics[scale=0.7]{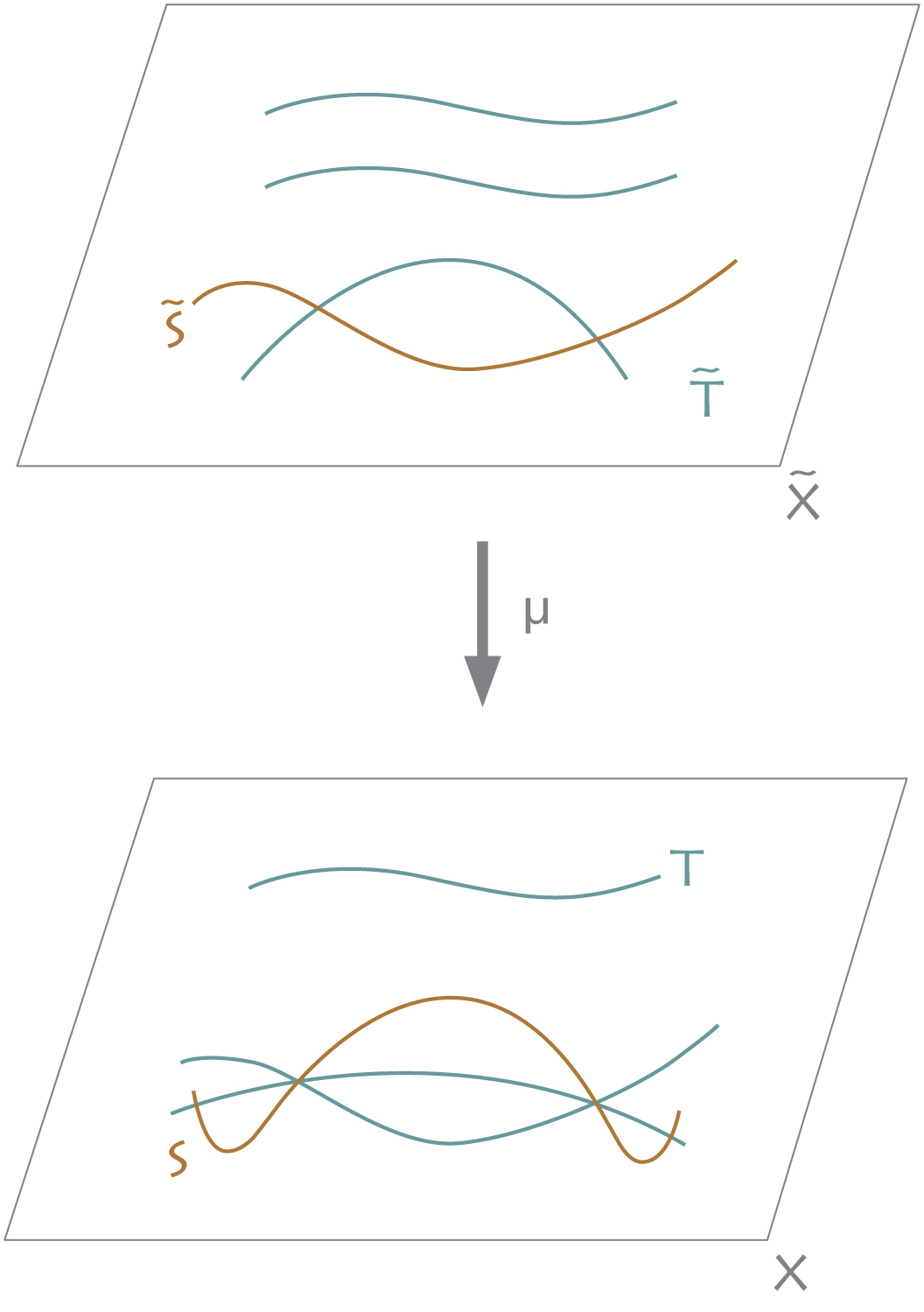}
\end{center}
\caption{The morphism $\mu\colon\ti{X}\to X$.}
\end{figure}


\bibliographystyle{alpha}

\end{document}